\documentclass[letterpaper]{article}

\usepackage[english]{babel}

\usepackage[letterpaper, margin=1in]{geometry}

\usepackage{amsmath,amssymb,amsfonts,amsthm}
\usepackage{mathrsfs}
\usepackage{bm}     
\usepackage{graphicx} 
\usepackage{subcaption}
\usepackage{indentfirst}
\usepackage{xspace}
\usepackage{hyperref} 
\usepackage{url}   

\usepackage{algorithm}
\usepackage{algorithmic}

\usepackage{appendix}
\usepackage{booktabs}   
\usepackage{multirow}
\usepackage{longtable}
\usepackage{cases}
\usepackage{pdfpages}
\usepackage{pdflscape}
\usepackage{empheq}
\usepackage{dsfont}    
\usepackage{extarrows}
\newtheorem{theorem}{Theorem}
\newtheorem{lemma}{Lemma}
\newtheorem{proposition}{Proposition}

\newtheorem{remark}{Remark}
\newtheorem{example}{Example}
\newcommand{\vecop}{\text{vec}}

\newcommand{\grad}{\operatorname{grad}}
\usepackage{xspace}

\newcommand{\W}{\text{Wasserstein}\xspace}

\providecommand{\keywords}[1]{\textbf{\textit{Keywords---}} #1}

\title{Accelerated Natural Gradient Method for Parametric Manifold Optimization%
  \thanks{The first two authors contributed equally to this work. This work was supported in part by the National Natural Science Foundation of China under the grant numbers 12331010 and 12288101, and National Key Research and Development Program of China under the grant number 2024YFA1012903.}}

\author{%
  Chenyi Li\thanks{School of Mathematical Sciences, Peking University, China
    ({lichenyi@stu.pku.edu.cn}).}%
  \and
  Shuchen Zhu\thanks{Center for Data Science, Peking University, China
    ({shuchenzhu@stu.pku.edu.cn}).}%
  \and
  Zhonglin Xie\thanks{Beijing International Center for Mathematical Research,
    Peking University, China
    ({zlxie@pku.edu.cn}).}%
  \and
  Zaiwen Wen\thanks{Corresponding author. Beijing International Center for Mathematical Research,
    Peking University, China
    ({wenzw@pku.edu.cn}).}
}

\begin{document}
\maketitle
\begin{abstract}
Parametric manifold optimization problems frequently arise in various machine learning tasks, where state functions are defined on infinite-dimensional manifolds. We propose a unified accelerated natural gradient descent (ANGD) framework to address these problems. By incorporating a Hessian-driven damping term into the manifold update, we derive an accelerated Riemannian gradient (ARG) flow that mitigates oscillations. An equivalent first-order system is further presented for the ARG flow, enabling a unified discretization scheme that leads to the ANGD method. In our discrete update, our framework considers various advanced techniques, including least squares approximation of the update direction, projected momentum to accelerate convergence, and efficient approximation methods through the Kronecker product. It accommodates various metrics, including $H^s$, Fisher-Rao, and Wasserstein-2 metrics, providing a computationally efficient solution for large-scale parameter spaces. We establish a convergence rate for the ARG flow under geodesic convexity assumptions. Numerical experiments demonstrate that ANGD outperforms standard NGD, underscoring its effectiveness across diverse deep learning tasks.
\end{abstract}

\keywords{Natural gradient, Parametric manifold, Riemannian optimization, Accelerated flow, Machine learning}

\section{Introduction}
We focus on the parametric  optimization problem of the form
\begin{align}
    \min_{\theta \in \mathbb{R}^p} L(\rho_\theta),
    \label{problem: form}
\end{align}
where \(\rho_{(\cdot)}: \Theta \to \mathcal{M}\) is a mapping from the parameter space \(\Theta \subseteq \mathbb{R}^p\) to the state space \(\mathcal{M}\), which inherently exhibits the structure of an infinite-dimensional Riemannian manifold. The loss functional is given as \(L(\cdot): \mathcal{M} \to \mathbb{R}\). Classical Riemannian optimization treats \(\theta\) as residing on a specific manifold, while in parametric manifold optimization, \(\rho\) lies on the manifold \(\mathcal{M}\) and is parameterized by \(\theta\). Typically, \(\rho_\theta\) is produced by a deep neural network with input \(x \in \mathbb{R}^d\) and parameters \(\theta \in \mathbb{R}^p\). This framework encompasses a wide range of machine learning and data analysis problems, where \(\mathcal{M}\) can be specialized to function spaces such as Sobolev spaces \(H^s\) or probability measures \(\mathcal{P}(\mathbb{R}^p)\). Two classic examples of this problem are the physics-informed neural networks (PINNs) \cite{Raissi2019Physics, Raissi2018deep} and variational Monte Carlo (VMC) methods \cite{pfau2020ab} for solving Schrödinger equations.   

\subsection{Literature Review}
Gradient flow is helpful in the analysis and design of accelerated optimization algorithms. The continuous limit of the Nesterov's acceleration method \cite{nesterov1983method} is known as the accelerated gradient flow \cite{su2016ode}. A Hessian-driven damping term is incorporated into high-resolution ODE in \cite{Shibin2021understanding}. A convergence analysis of accelerated ODE with Hessian-driven damping is provided in \cite{Attouch2022}. The extension of accelerated gradient flow on Riemannian manifolds has been extensively studied. Gradient flow on the manifold is introduced to analyze the Riemannian Nesterov accelerated gradient method in \cite{alimisis2020continuous}. Accelerated information gradient flow \cite{wang2022accelerated} extends it to infinite-dimensional Riemannian manifolds. Minimization of the KL divergence using the Nesterov acceleration method from a continuous perspective is explored in \cite{mA2021Is}. The \W-2 gradient flow and its parametrization are studied in \cite{JIN2025parameterized, wu2023parameterizedwassersteinhamiltonianflow}.

Implementing the traditional gradient descent method for high-dimensional nonconvex problems \eqref{problem: form} is a significant challenge. For PDE-based optimization problems, traditional gradient descent algorithms suffer from pathologies \cite{understanding2021wang, pmlr-v235-rathore24a}. Natural gradient methods are often used to enhance the performance of the gradient descent algorithm by incorporating local curvature information, leading to faster convergence \cite{Martens2020Newinsights, yang2023NGP, pmlr-v206-bahamou23a}. The natural gradient method pulls back the curvature of $\theta$ to the manifold \cite{Amari1998Natural}. This adjustment allows for more efficient search directions, overcoming some of the limitations of traditional gradient descent algorithm, such as slow convergence or getting stuck in local minima. Several modifications of the natural gradient method have been explored. The adaptively regularized natural gradient method \cite{Wu2024Convergence} introduces an adaptive damping term. General natural gradient descent methods for PDE-based problems are discussed in \cite{nurbekyan2023efficient}, which formulates least squares problems for general metrics. The energy natural gradient method \cite{Achieving2023Munk} is proposed to accelerate the PINNs training from the perspective of an energy metric. 

Recent advancements in natural gradient methods for probability distributions have largely focused on the Fisher-Rao and \W-2 metrics. The computation of the natural gradient direction for the Fisher-Rao metric is reformulated as a least squares problem in \cite{chen2024empowering}. This approach is enhanced by incorporating the projected momentum \cite{Goldshlager2024spring} (also known as the Kaczmarz method) to refine the solution. In the realm of \W-2 natural gradient methods, a computationally efficient modification of the \W information matrix is introduced by \cite{JIN2025parameterized}. Additionally, Arbel et al. estimate the natural gradient by solving a regularized mini-max problem \cite{arbel2019kernelized}, whereas the KL divergence is used to approximate  the \W-2 natural gradient \cite{neklyudov2023wasserstein}, thereby avoiding the intractable computation of the \W information matrix.

Various approximation methods have been proposed to reduce the computational cost of the natural gradient methods. KFAC \cite{Martens2015KFAC, pmlr-v48-grosse16} approximates the layer-wise information matrix using the Kronecker decomposition, allowing its inverse to be computed efficiently as the product of the inverses of these smaller matrices.  Sketch-based methods \cite{Yang2022Sketch} reduce the computational cost of matrix multiplication and inversion by sampling rows or columns, enabling efficient approximations of natural gradient. The quasi-natural gradient method \cite{He2022QuasiNG} integrates LBFGS with natural gradient descent method to improve efficiency in statistical learning problems. Layer-wise block-diagonal approximations \cite{pmlr-v206-bahamou23a} are used to approximate the information matrix.

% The whole optimization process of the parametric manifold optimization problem \eqref{problem: form} can be decomposed into two stages. The first stage is to find a good local update point on the manifold, and the second stage is to project the point on the manifold onto the parameter space. Acceleration methods can be applied for the first stage. Precise and efficient approximations are introduced to solve for the projection. Finally we obtain our accelerated natural gradient descent (ANGD) method. 

\subsection{Our Contributions}
In this paper, we propose a new accelerated natural gradient method to solve \eqref{problem: form}. Our main contributions are as follows.
\begin{itemize}
% \item 
% We extend the accelerated gradient flow with Hessian-driven damping in Euclidean spaces to infinite-dimensional Riemannian manifolds, resulting in the accelerated Riemannian gradient (ARG) flow for the state variable. Theoretical analysis shows that the ARG flow achieves a convergence rate of $\mathcal{O}(t^{-2})$ under geodesic convexity, strictly generalizing Euclidean acceleration mechanisms while maintaining their characteristic decay properties.

\item We derive the natural gradient descent algorithm from a novel two-stage optimization process. The first stage identifies an update flow on the manifold, while the second projects this flow onto the parameter space for a discrete update scheme. This separation allows for distinct designs of the natural gradient method in the manifold and parameter space. Though taking the metric of the manifold into consideration \cite{Martens2020Newinsights}, traditional natural gradient methods do not seek acceleration for the first stage. We introduce an accelerated Riemannian gradient (ARG) flow to address this. A Hessian-driven damping term, designed to mitigate oscillations and accelerate convergence, is incorporated into the analysis of manifold-accelerated gradient flows for the first time. This ODE generalizes different first-order acceleration methods with different choices of hyper-parameters. Theoretical analysis shows that the ARG flow achieves a convergence rate of $\mathcal{O}(t^{-2})$ under geodesic convexity, strictly generalizing Euclidean acceleration mechanisms while maintaining their characteristic decay properties.

\item We present a novel algorithmic framework that unifies projection across different metric spaces in the second stage. This framework discretizes the update quantity of the ARG flow in the tangent space of the manifold and maps it to the parameter space, resulting in a unified accelerated natural gradient descent (ANGD) method. Consequently, the ANGD method achieves faster convergence compared to traditional natural gradient descent methods \cite{nurbekyan2023efficient}. Our approach generalizes several established methods, including the least squares approximation \cite{Yang2022Sketch}, projected momentum \cite{Goldshlager2024spring}, and KFAC \cite{Martens2015KFAC}. These methods provide increased flexibility in selecting preconditioners and eliminate the need to estimate metric-specific information matrices, thus improving computational efficiency, particularly for metrics with intractable information matrices, such as the $H^s$ metric (for $s<0$) and the \W-2 metric. Numerical experiments demonstrate the substantial acceleration of the ANGD method over non-accelerated natural gradient methods.

% \item For the second stage, we propose a unified discretization scheme to project the ODE flow onto the parameter space. We design different methods depending on whether the first-order variation of $L$ with respect to $\rho$ is attainable. Using a first-order approximation of the update, we derive a linear system for the accelerated direction of the parameters. Various methods are employed to efficiently solve this system, leading to the development of our accelerated natural gradient method. The accelerated flow gives the ability for the discrete scheme to achieve faster convergence compared to standard natural gradient methods \cite{nurbekyan2023efficient}. We present the corresponding update schemes for four specific metrics and apply these schemes to concrete problems.  Update algorithms for other metrics can be derived similarly. Numerical experiments on PINNs and VMC show that ANGD accelerates optimization compared to other natural gradient or first-order methods. 
\end{itemize}

\subsection{Notation}
We use $\left\langle \cdot, \cdot \right\rangle$ to denote the inner product in Euclidean space. The spatial and parameter gradients are represented by $\nabla$ and $\partial_\theta$, respectively. For a time-varying map and input, we abbreviate $\partial_t f_t(x) \big|_{x=x_t}$ as $\partial_t f_t(x_t)$. The divergence of a vector-valued function $F$ is written as $\nabla \cdot F$. The symbol $\circ$ is used to indicate the composition of two maps. In integrals, we omit the differential notation $dx$ when there is no risk of ambiguity. We utilize $x[i]$ to refer to the $i$-th coordinate of the vector $x$. Given a vector $v \in \mathbb{R}^n$, its center value is denoted by $\overline{v}=v-\sum_{i=1}^n v[i]/n$, where the subtraction is applied element-wise.
%We adopt the following notation conventions: the subscript $k$ denotes quantities at the $k$-th iteration, while the subscript $t$ indicates quantities at time $t\ge0$ with respect to the ODE. 

\subsection{Organization}
The rest of this paper is organized as follows. In Section \ref{sec: continuous}, we introduce the ARG flow with Hessian drive damping. The general discretization scheme and approximation methods on the parameter space are discussed in Section \ref{sec: discrete}. Four kinds of specific metrics are considered in Section \ref{sec: Examples}. In Section \ref{sec: theory}, we give the convergence analysis for the ARG flow. Finally, we show the performance of the ANGD algorithms with different numerical experiments in Section \ref{sec: experiments}.

\section{A Continuous-time Model for Accelerating NGD} \label{sec: continuous}
In this section, we derive an accelerated gradient flow featuring Hessian-driven damping on the manifold $\mathcal{M}$. The trajectory of this flow exhibits a faster convergence rate and improved convergence properties compared with gradient flow. This continuous-time model serves as an ideal template for developing accelerated natural gradient descent methods on manifolds. Before proceeding, we provide a brief review of some basic concepts related to the Riemannian metric.

% Considering the first stage, we want to find a good path for $\rho$ on the manifold space $\mathcal{M}$. Intuitively, just by following this path well, we can get a good path for the parameters to update. From this point of view, we can easily derive some acceleration algorithms just by finding better local update points on the manifold. We use the gradient flow system to describe the update of $\rho$.

\subsection{Background on Riemannian metric}

We begin by outlining the definition of Riemannian metrics on an (infinitely dimensional) Riemannian manifold.
Let $\mathcal{M}$ be a set of functions that are defined on $\Omega$, a region in $\mathbb{R}^d$. The tangent space of $\mathcal{M}$ at $\rho(x) \in\mathcal{M}$ is defined as
\[
T_\rho\mathcal{M}=\{\partial_t \rho_t(x)|_{t=0}: \rho_t(x):(-1,1)\times \Omega \to \mathbb{R} \text{ is a smooth curve in }\mathcal{M} \text{ and } \rho_0=\rho\}.
\]
In this section, we omit the variable $x$ of the $\rho$ function when there is no risk of ambiguity. The cotangent space $T_\rho^*\mathcal{M}$ is the dual space of $T_\rho\mathcal{M}$, consisting of linear functionals acting on $T_\rho\mathcal{M}$ defined via the  $L^2$ inner product. The metric tensor $\mathcal{G}(\rho):T_\rho\mathcal{M}\to T_\rho^*\mathcal{M}$ is an invertible and Hermitian linear operator. It satisfies the following properties: 1) $\int \sigma_1 \mathcal{G}(\rho) \sigma_2 dx= \int \sigma_2 \mathcal{G}(\rho) \sigma_1 dx$ for all $\sigma_1, \sigma_2 \in T_\rho \mathcal{M}$; 2) $\int \sigma \mathcal{G}(\rho) \sigma dx\ge 0$ for all $\sigma \in T_\rho \mathcal{M}$, with equality holding if and only if $\sigma\equiv0$.
We then define the Riemannian metric on $T_\rho\mathcal{M}$ as 
\begin{align*}
    g_\rho(\sigma_1, \sigma_2) = \int \sigma_1 \mathcal{G}(\rho) \sigma_2 dx= \int \Phi_1 \mathcal{G}(\rho)^{-1} \Phi_2 dx, \; \forall \sigma_1, \sigma_2 \in T_\rho \mathcal{M},
\end{align*}
where $\Phi_i=\mathcal{G}(\rho) \sigma_i$. Given a metric, we can define the Riemannian distance between two elements $\rho_0,\rho_1\in \mathcal{M}$ as  $\text{dist}(\rho_0,\rho_1)^2=\inf \int_0^1 g_{\rho_t}(\partial_t\rho_t,\partial_t\rho_t)dt$, where the infimum is taken over all smooth curves ${\rho_t:[0,1]\to \mathcal{M}}$ connecting $\rho_0$ and $\rho_1$. 

Then we extend the concept of the Riemannian gradient from finite-dimensional Riemannian manifolds to the infinite-dimensional case following \cite{riemannianhilbert}. We introduce some basic notation about calculus variations. The $L^2$ first variation for a functional $L(\rho) : L^2\to \mathbb{R}$ with respect to  $\rho$ is as $\frac{\delta L}{\delta \rho}$, which is defined as the function such that 
$\lim\limits_{\epsilon\to 0}\frac{L(\rho+\epsilon\sigma)-L(\rho)}{\epsilon}=\int \frac{\delta L}{\delta \rho} \sigma,$
holds for any $\sigma\in L^2$. 
For any given $\rho_0\in L^2$, we denote $\left.\frac{\delta L}{\delta \rho}\right|_{\rho=\rho_0}$ by $\frac{\delta L}{\delta \rho_0}$ when there is no ambiguity. The first order variation with respect to a function $h(\rho)(x):L^2\times\mathbb{R}^d\to \mathbb{R}$ is given as
$\frac{\delta h}{\delta \rho}(x,y)=\frac{\delta }{\delta \rho}\int h(y)\delta(x-y)dy,$
where $\delta$ is the Dirac delta function.   
We denote the canonical action of a tangent field $h(\rho)(x)$ to a smooth functional $L(\rho)$ as $h\circ L(\rho)=\int \frac{\delta E}{\delta \rho}(x) \cdot h(\rho)(x) dx.$

Further, we define the directional variational derivative of a smooth linear mapping $\mathcal{A}(\rho):T_\rho \mathcal{M}\to T_\rho^* \mathcal{M}$. For any $\sigma\in T_\rho \mathcal{M}$, the directional variational derivative $\frac{\partial \mathcal{A}(\rho)}{\partial \rho}\cdot \sigma:T_\rho \mathcal{M}\to T_\rho^* \mathcal{M}$ is defined as
$
    \left[\frac{\partial \mathcal{A}(\rho)}{\partial \rho}\cdot \sigma\right]\tau =\lim_{\epsilon\to0} \frac{ \mathcal{A}(\rho+\epsilon\sigma)\tau-\mathcal{A}(\rho)\tau}{\epsilon},
$
for any fixed $\tau\in T_\rho \mathcal{M}$. This definition can be naturally extended to smooth linear mappings $\mathcal{A}(\rho):T_\rho^* \mathcal{M}\to T_\rho \mathcal{M}$. The Riemannian gradient of a smooth functional 
$L(\rho)$ on $\mathcal{M}$ is given as $\operatorname{grad} L(\rho)=\mathcal{G}(\rho)^{-1} \frac{\delta L}{\delta \rho}.$

The objects $\rho$ in this paper fall into two types, with their tangent and cotangent spaces exhibiting different properties.
\begin{enumerate}
    \item \textbf{$\rho$ as a PDE-based model.} In this case, we primarily consider $\rho \in H^s(\Omega)$ (with $s$ being an integer), where $\Omega$ is a bounded open domain or the entire Euclidean space $\mathbb{R}^d$. The Sobolev space $\mathcal{M}=H^s(\Omega)$ is a Hilbert space, with the tangent space $T_\rho\mathcal{M}=H^s(\Omega)$  and the cotangent space  $T_\rho^*\mathcal{M}=L^2(\Omega)$.
    \item \textbf{$\rho$ as a probability distribution.} Define the set of smooth probability densities on $\Omega$ as $\mathcal{M} = \mathcal{P}(\Omega) = \{\rho \in \mathcal{C}^{\infty}(\Omega) : \int_{\Omega} \rho  dx = 1, \, \rho \geq 0\}$, where $\Omega$ is an open set in $\mathbb{R}^d$. Here, $\mathcal{C}^{\infty}(\Omega)$ denotes the set of smooth functions defined on $\Omega$. In this case, it holds that 
        $T_\rho \mathcal{M}=\left\{\sigma \in \mathcal{C}^{\infty}(\Omega): \int \sigma d x=0\right\}$, $T_\rho^* \mathcal{M}=\mathcal{C}^{\infty}(\Omega) / \mathbb{R}$.
    For this case, we mainly consider the Fisher-Rao  and  \W-2 metric.
\end{enumerate}

\subsection{Accelerated Gradient Flow in Riemannian Manifolds}
We start with an optimization problem in finite-dimensional Euclidean space:
\begin{equation}
  \min\limits_{\varrho \in \mathbb{R}^p} \ell(\varrho), 
\end{equation}
where $\ell$ is a differentiable function. The gradient descent method updates 
$\varrho$ by following the steepest update direction, with the continuous counterpart described by $\dot{\varrho}_t +\nabla \ell(\varrho_t)=0$.  
Nesterov proposed an accelerated method (NAG) to improve the efficiency of gradient descent method for convex functions \cite{nesterov1983method}. A second-order ODE has been proposed to elucidate the acceleration mechanism corresponding to NAG in \cite{su2016ode} as $\ddot{\varrho}_t +\alpha_t \dot{\varrho}_t +\nabla \ell(\varrho_t)=0.$
% \begin{equation}\label{acc:flow}
%  \ddot{\varrho}_t +\alpha_t \dot{\varrho}_t +\nabla \ell(\varrho_t)=0.  
% \end{equation}
A Hessian-driven damping term is introduced to the continuous Nesterov acceleration ODE by \cite{attouch2020fastconvexoptimizationinertial,Attouch2022}, which takes the form:
\begin{equation}\label{ISHD:flow}
 \ddot{\varrho}_t +\alpha_t \dot{\varrho}_t +\beta_t \nabla^2 \ell(\varrho_t)\dot{\varrho}_t+\gamma_t\nabla \ell(\varrho_t)=0,
\end{equation}
where $\alpha_t, \beta_t, \gamma_t$ are non-negative functions only depending on $t$. The Hessian-driven damping term $\nabla^2 \ell(\varrho_t)\dot{\varrho}_t$ facilitates faster convergence and reduces oscillations.

We extend these concepts to Riemannian manifolds. The Riemannian gradient flow of the target functional $L$ on Riemannian manifold $\mathcal{M}$ is defined as 
\begin{align}
    \partial_t \rho_t = - \text{grad} L(\rho_t)=-\mathcal{G}(\rho)^{-1} \frac{\delta L}{\delta \rho},
    \label{eq: vanilla gradient flow}
\end{align}
where $\rho_t : [t_0, \infty) \to \mathcal{M}$ is a smooth curve on the manifold. For the accelerated gradient flow on Riemannian manifolds, note that $\ddot{\varrho}_t=\frac{d}{dt}\frac{d}{dt}{\varrho}_t$ in the accelerated gradient flow \eqref{ISHD:flow} represents taking directional derivatives of $\frac{d}{dt}{\varrho}_t$ along $\frac{d}{dt}{\varrho}_t$. It holds for the damping term that $\nabla^2 f(\varrho_t)\dot{\varrho}_t= \frac{d}{dt}\nabla f(\varrho_t)$. Hence, by replacing the second-order term $\frac{d}{dt}\frac{d}{dt}{\varrho}_t$ and $\frac{d}{dt} \nabla f(\varrho_t)$ with the Levi-Civita connection and the Riemannian gradient as $\nabla_{\partial_t \rho_t} \partial_t \rho_t$ and $\nabla_{\partial_t \rho_t} \grad L(\rho_t)$, we obtain the accelerated Riemannian gradient flow with Hessian-driven damping:
\begin{equation}\label{R_acc:flow}
 \nabla_{\partial_t \rho_t} \partial_t \rho_t+\alpha_t \partial_t \rho_t +\beta_t \nabla_{\partial_t \rho_t} \grad L(\rho_t)+\gamma_t\grad L(\rho_t)=0.   
\end{equation}
Note that Levi-Civita connection in \eqref{R_acc:flow} is not straightforward to discrete with respect to time $t$. To derive the appropriate first order ODE formulation for numerical discretization, which only contains first order derivative to time and space, we apply the following transformation.
\begin{proposition} \label{prop: Riemannian system}
Define $\Phi_t=\Psi_t-\beta_t \frac{\delta L}{\delta \rho_t}$ and  the Riemannian correction term 
\[
R_t= \frac{1}{2} \left.\frac{\delta}{\delta \rho} \left[ \int \Phi_t \mathcal{G}(\rho)^{-1} \Phi_t \, dx \right]\right|_{\rho=\rho_t}+\frac{\beta_t}{2}\mathcal{G}(\rho_t)\left[\frac{\partial (\mathcal{G}(\rho_t)^{-1})}{\partial \rho_t}\cdot \mathcal{G}(\rho_t)^{-1}\Phi_t\right]\frac{\delta L}{\delta \rho_t}.
\] 
The second-order accelerated Riemannian gradient flow \eqref{R_acc:flow} is equivalent to 
% \begin{align}
% \begin{cases}
%     &\partial_t \rho_t - \mathcal{G}(\rho_t)^{-1} \Phi_t = 0, \\
%     &\partial_t \Phi_t+\alpha_t \Phi_t + \frac{1}{2} \frac{\delta}{\delta \rho_t} \left( \int \Phi_t \mathcal{G}(\rho_t)^{-1} \Phi_t \, dx \right) \\&\quad\quad\quad\quad\quad+\frac{\beta_t}{2}\partial_t\frac{\delta L}{\delta \rho_t}+\frac{\beta_t}{2}\mathcal{G}(\rho_t)\partial_t\operatorname{grad}\, L(\rho_t) +\gamma_t\frac{\delta L}{\delta \rho_t} = 0,
% \end{cases}
%     \label{ode : accelerate_1}
% \end{align}
\begin{numcases}{}
    \partial_t \rho_t - \mathcal{G}(\rho_t)^{-1} \left( \Psi_t - \beta_t \frac{\delta L}{\delta \rho_t} \right) = 0, \label{ode : accelerate} \\
    \partial_t \Psi_t + \alpha_t \Psi_t + R_t + \left( \gamma_t - \dot{\beta}_t - \alpha_t \beta_t \right) \frac{\delta L}{\delta \rho_t} = 0, \label{ode : accelerate_2}
\end{numcases}
with initial values $\rho_0$ and $\Psi_0=\beta_0\frac{\delta L}{\delta \rho_0}$. 
\end{proposition}
The proof of this proposition can be found in Section \ref{sec: theory-riemannian system}.

By incorporating the specific metric $\mathcal{G}$ into equation \eqref{ode : accelerate} and \eqref{ode : accelerate_2}, we can derive the accelerated flow corresponding to each concrete metric. We refer the reader to \cite{nurbekyan2023efficient} and \cite{wang2022accelerated} for foundational information on the metrics.

\begin{example}[$L^2$ ARG flow]
Note that $L^2=H^0$ is a special case of $H^s$ space. Since in $L^2$ space, $\mathcal{G}(\rho_t) = Id$, the ARG flow for $(\rho_t,\Psi_t)$ in \eqref{ode : accelerate} can be greatly simplified as:
\begin{align}\label{ode : L2}
    \begin{cases}
        \partial_t \rho_t = \Psi_t - \beta_t \frac{\delta L}{\delta \rho_t} , \\
        \partial_t \Psi_t + \alpha_t \Psi_t -(\alpha_t\beta_t - \gamma_t +  \dot{\beta}_t) \frac{\delta L}{\delta \rho_t}= 0.
    \end{cases}
\end{align}
\end{example}

\begin{example}[$H^s$ ($s\ge0$) ARG flow]
$H^s$ ($s\ge0$) ARG flow satisfies
\begin{equation}\label{eq: Hspos ode}
\left\{\begin{array}{l}
\partial_t \rho_t-\left[\sum_{i=0}^{s}(-\Delta)^i \right]^{-1}\left(\Psi_t - \beta_t \frac{\delta L}{\delta \rho_t}\right) =0, \\
\partial_t \Psi_t + \alpha_t \Psi_t -(\alpha_t\beta_t - \gamma_t +  \dot{\beta}_t) \frac{\delta L}{\delta \rho_t}= 0.
\end{array}\right. 
\end{equation}
\end{example}

\begin{example}[$H^s$ ($s<0$) ARG flow]
$H^s$ ($s<0$) ARG flow satisfies
\begin{equation}\label{eq: Hsneg ode}
\left\{\begin{array}{l}
\partial_t \rho_t-\sum_{i=0}^{|s|}(-\Delta)^i\left(\Psi_t - \beta_t \frac{\delta L}{\delta \rho_t}\right)  =0, \\
\partial_t \Psi_t + \alpha_t \Psi_t -(\alpha_t\beta_t - \gamma_t +  \dot{\beta}_t)\frac{\delta L}{\delta \rho_t}= 0.
\end{array}\right. 
\end{equation}
\end{example}

\begin{example}[Fisher-Rao ARG flow]
Fisher-Rao ARG flow satisfies
\begin{equation}
\label{Fisher-Rao:flow}
\left\{
\begin{aligned}
    &\partial_t \rho_t - \left(\Psi_t - \beta_t \frac{\delta L}{\delta \rho_t} - \mathbb{E}_{\rho_t}\left[\Psi_t - \beta_t \frac{\delta L}{\delta \rho_t}\right]\right) \rho_t = 0, \\
    &\partial_t \Psi_t + \alpha_t \Psi_t 
    + \frac{1}{2} \left(\Phi_t
    - \mathbb{E}_{\rho_t}\left[\Phi_t \right]\right)\left(\Psi_t  
    - \mathbb{E}_{\rho_t}\left[\Psi_t\right]\right)   
    -(\alpha_t\beta_t - \gamma_t +  \dot{\beta}_t)  \frac{\delta L}{\delta \rho_t} = 0.
\end{aligned}
\right.
\end{equation}
\end{example}

\begin{example}[\W ARG flow]
\W ARG flow satisfies
\begin{equation}\label{W2:flow}
\left\{\begin{array}{l}
\partial_t \rho_t+\nabla\cdot \left(\rho_t\nabla \left(\Psi_t - \beta_t \frac{\delta L}{\delta \rho_t}\right)\right) =0, \\
\partial_t \Psi_t+\alpha_t \Psi_t+\frac{1}{2} \left\|\nabla \Phi_t\right\|^2+\frac{\beta_t}{2}w_t-(\alpha_t\beta_t - \gamma_t +  \dot{\beta}_t)\frac{\delta L}{\delta \rho_t}=0 ,
\end{array}\right. 
\end{equation}
where $w_t$ is a solution of $\nabla\cdot\left(h_t\nabla\frac{\delta L}{\delta \rho_t}-\rho_t\nabla w_t\right)=0$ with   $h_t=\nabla\cdot(\rho_t\nabla \Phi_t)$.
\end{example}

\section{A Discretization Scheme on the Parameter Space}\label{sec: discrete}
Along the trajectory of the ARG flow \eqref{R_acc:flow}, $L$ converges to its minimum. However, the discrete updates are performed in terms of the parameters $\theta \in \Theta$. Therefore, we need to map the iterations from the manifold back to the parameters. For some complicated metrics, we transform the original ODE of \(\partial_t \rho_t(x)\) in \eqref{ode : accelerate} as follows
\begin{equation}\label{tran-rho}
   \partial_t\mathcal{A}(\rho_t)(x)=u_t(x),
\end{equation}
where $\mathcal{A}$ is a metric-specific map with function output, $u_t(x)$ is obtained based on $\mathcal{A}$ and the original $\mathcal{G}(\rho_t)^{-1} \left( \Psi_t - \beta_t \frac{\delta L}{\delta \rho_t} \right)$ in \eqref{ode : accelerate}. The purpose of this transformation is to adapt to various problem structures. For example, when dealing with a normalized probability density $\rho_t(x)=\psi_t(x)/\int \psi_t(x) dx$, the quantity
$\partial_t \log \rho_t(x)=\partial_t \log \psi_t(x)-\mathbb{E}_{\rho_t}[\partial_t \log \psi_t]$ can be easily estimated, while
$\partial_t \rho_t(x)$ is not directly accessible. Specific examples of \eqref{tran-rho} are provided in Table \ref{table: a and ut}. 
%Definitions of $\mathcal{S}(\rho_t)$ to $q_t$ will be explained in Section \ref{sec-b-unavailable}. 
\begin{table}[htbp]
    \renewcommand{\arraystretch}{2.5}
    \centering
    \resizebox{\textwidth}{!}{
    \begin{tabular}{|c|c|c|c|c|c|c|}
        \hline
          Metric& $\mathcal{G}(\rho_t) ^{-1}\Phi_t$&$\mathcal{A}(\rho_t)$&$\mathcal{S}(\rho_t)$&$\mathcal{B}(\rho_t)$&$u_t$&$q_t$\\
        \hline
        Fisher-Rao    &$\left(\Phi_t-\int \Phi_t \rho_t d x\right)\rho_t$&$\log \rho_t$&$\log \rho_t$&$\log \rho_t$&$-\left(\Phi_t-\int \Phi_t \rho_t d x\right)$&$\rho_t$\\
        \hline
         \W-2 &$-\nabla  \cdot ( \rho_t  \nabla \Phi_t) $&$\log \rho_t$&$\log \rho_t$&$\log \rho_t$&$-\left\langle\nabla \text{log} \rho_t, \nabla \Phi_t\right\rangle-\Delta\Phi_t$&$\rho_t$\\
        \hline
       $H^s$ ($s\in \mathbb{N}$)    &$ \left[\sum_{i=0}^s(-\Delta)^i\right]^{-1} \Phi_t$&$ \rho_t$&$ \left[\sum_{i=0}^s(-\Delta)^i\right]\rho_t$&$\mathbf{D}^s \rho_t$&$-\left[\sum_{i=0}^s(-\Delta)^i\right]^{-1} \Phi_t$&Leb. 
       \\
        \hline
        $H^s$ ($s\in \mathbb{Z}^{-}$)  &$\left[\sum_{i=0}^{|s|}(-\Delta)^i\right] \Phi_t$&$ \rho_t$&$ \rho_t$&$\rho_t$&$-\left[\sum_{i=0}^s(-\Delta)^i\right]\Phi_t$&Leb.
       \\
        \hline
    \end{tabular}
    }
    \caption{Examples of metric-specific transformed ARG flow \eqref{tran-rho}. Leb. denotes the Lebesgue measure in  Euclidean spaces.} 
    \label{table: a and ut}
\end{table}

\subsection{A Unified Discretization Scheme}\label{sec: general discretized scheme}
%General Form of the Discretized Scheme
In this subsection, we discuss a unified discretization scheme for \eqref{tran-rho} and \eqref{ode : accelerate_2}. To clarify the notation, we use \(t\) as a subscript for continuous-time variables, \(k\) as a subscript for discrete-time variables. The core procedure is as follows. We obtain the first-order system corresponding to \(\partial_t \rho_t(x)\) in~\eqref{tran-rho} and \(\partial_t \Psi_t(x)\) in~\eqref{ode : accelerate_2}. Introducing parameters \(\theta\) into the update of \(\rho_t(x)\), it yields a parameterized ODE where \(\rho_{\theta_t}(x)\) approximates \(\rho_t(x)\). We denote the update direction for $\theta_k$ at the $k$-th iteration as $d_k$, which can be seen as an approximation of continuous parameter update $\partial_t \theta_t$ at time $t_k$. Denote the step size at the $k$-th iteration as $h_k$. The update rule for $\theta_k$ is given as
\begin{align}\label{eq: parameter update}
    \theta_{k+1} = \theta_k + h_k d_k.
\end{align}
Sampling both sides of the parameterized ODE for \(\rho_{\theta_t}(x)\) produces a linear system for the update direction \(d_k\). Moreover, we derive the update rule for \(\Psi_t(x)\) from a direct discretization, which is used to update the right-hand side of the linear system.  
We summarize the associated notations in Table~\ref{table: notations for general}.
\begin{table}[htbp]
\centering
\begin{tabular}{|c|c|c|c|} 
\hline
 & Continuous Time & Parametrized   & Discrete Time  \\ \hline
 Parameters & not occurred & $\theta_t$ & $\theta_k$  \\ \hline
 Update Direction & not occurred & $\partial_t \theta_t$ & $d_k$\\ \hline
 $\rho$-manifold& $\rho_t(x)$ & $\rho_{\theta_t} (x)$ & $\rho_{\theta_k}(x)$ \\ \hline
 $\Psi$-momentum & $\Psi_t(x)$& not used  & $\Psi_k(x)$ \\ \hline
 Riemannian correction term & $R_t(x)$ &not used & $R_k(x)$ \\ \hline
\end{tabular}
\caption{Different notations used for the discretized scheme.}
\label{table: notations for general}
\end{table}

We next discuss how to form a linear system to get $d_k$. By the chain rule, it yields
\begin{align}
    \partial_t \mathcal{A} (\rho_{\theta_t})(x) = (\partial_\theta \mathcal{A}(\rho_{\theta_t})(x))^\top \partial_t \theta_t ,
\end{align}
where $\partial_\theta \mathcal{A}(\rho_{\theta_t})(x) = [\partial_{\theta[1]} \mathcal{A}(\rho_{\theta_t})(x), \cdots, \partial_{\theta[p]} \mathcal{A}(\rho_{\theta_t})(x) ]^\top$ are $p$ functions on the tangent space. %The inner product is taken as real number multiply on functions. 
Hence, we construct the following parameterized ODE:
\begin{align}\label{eq: theta continuous}    (\partial_{\theta} \mathcal{A}(\rho_{\theta_t})(x))^\top \partial_t \theta_t = u_t(x),\quad \forall \; x\in \mathbb{R}^d.
\end{align}
Equation \eqref{eq: theta continuous} shows that two functions are equal at the tangent space. We can only take a finite set of sample points $\{x^j_k\}_{j=1}^n$ to approximate these functions with matrices and vectors. Therefore, based on these samples, we build a linear system approximating \eqref{eq: theta continuous} to find a suitable update direction $d_k$. It can also be seen as the discretization of $\partial_t \theta_t$ at the $k$-th iteration. The right-hand side of \eqref{eq: theta continuous} becomes a real number for a specific sample. For the left-hand side, the function vector $\partial_\theta \mathcal{A}(\rho_{\theta_t})(x_k^j)$ turns into a vector in $\mathbb{R}^p$. Given the complexity of our networks, we can assume that the number of parameters $p$ exceeds the sample size $n$. This leads to an underdetermined system as follows:
\begin{align} \label{eq: single line of general system}
    (\partial_\theta \mathcal{A}(\rho_{\theta_k}) (x^j_k))^\top  d_k = u_k (x^j_k), \quad \forall \; 1\leq j \leq s.
\end{align}
For specific problems, the samples $\{x^j_k\}_{j=1}^n$ may be either fixed or changeable. 
% Furthermore, when $\rho(x)$ represents a probability density, we focus on $\log \rho(x)$. Similarly, we can derive a linear system as \eqref{eq: single line of general system} for the update direction $d_k$. Detailed explanations for this procedure will be provided for both the Fisher-Rao metric and the \W-2 metric.
For simplicity, we denote the following notations:
\begin{equation}
    \begin{aligned}
       O_k&=[\partial_{\theta} \mathcal{A}(\rho_{\theta_k})(x_k^1),...,\partial_{\theta} \mathcal{A}(\rho_{\theta_k})(x_k^n)]^\top\in\mathbb{R}^{n \times p} ,   \\ b_k&=\left[u_k(x_k^1),...,u_k(x_k^n)\right]^\top\in\mathbb{R}^{n\times 1}.
       \nonumber
    \end{aligned}
\end{equation}
The system \eqref{eq: single line of general system} can be rewritten as
\begin{align}\label{eq: general abstract system}
    O_k d_k = b_k.
\end{align}
The method of solving system \eqref{eq: general abstract system} varies depending on the metric and the optimization problem. We summarize some general approaches in Section \ref{sec: solution method of linear system}.

Now we focus on the update rule for $\Psi_k(x_k^j)$, which is used to update $b_k$. Since $\Psi_t(x)$ serves as an auxiliary variable in the ODEs \eqref{ode : accelerate} and \eqref{ode : accelerate_2}, we approximate $\partial_t \Psi_t(x)$ at the $k$-th step using a forward difference over the time interval $h_{k}$ at sample points as follows
\begin{align}\label{eq: Psi discrete}
    \partial_t \Psi_t(x_k^j) \big|_{t=t_k} \approx \frac{\Psi_k(x_k^j) - \Psi_{k-1}(x_k^j)}{h_{k}}.
\end{align}
For terms in \eqref{ode : accelerate_2} which are not related to the time derivative, we simply use their values at $t_k$. This gives us the update for $\Psi_k(x^j_k)$ as follows:
\begin{align}\label{eq: Psi general update}
    \Psi_k(x_k^j) = \mu_k \Psi_{k-1}(x_k^j) - h_{k} \left( R_k(x_k^j) + \left( \gamma_k - \dot{\beta}_k - \alpha_k \beta_k \right) \frac{\delta L}{\delta \rho_{\theta_k}}(x_k^j) \right), 
\end{align}
where $\mu_k = {1-h_{k}\alpha_k}$. The update scheme for $\Psi_k(x_k^j)$  may vary between different metrics, which will be thoroughly discussed when considering specific metrics.
% The solution method of \eqref{eq: general abstract system} differs from each metric and will be discussed separately in this section. 

In each iteration, \(\Psi_k(x_k^j)\) is computed using \eqref{eq: Psi general update}, which in turn determines \(b_k\). We then solve equation \eqref{eq: general abstract system} for the update direction \(d_k\), and update the parameters according to equation \eqref{eq: parameter update}. The general method of solving \eqref{eq: general abstract system} is discussed in Section \ref{sec: solution method of linear system}. For some cases, $b_k$ in \eqref{eq: general abstract system} is hard to get. We discuss methods to deal with these cases in Section \ref{sec-b-unavailable}. 

\subsection{Direct Methods for Solving \eqref{eq: general abstract system}}\label{sec: solution method of linear system}
As an underdetermined system, directly solving the pseudo-inverse of $O_k$ for \eqref{eq: general abstract system} by singular value decomposition is time-consuming. Without loss of generality, suppose that $O_k$ has full row rank. We propose three different algorithms to solve this system. We can multiply $O_k^\top $ on both sides of the system \eqref{eq: general abstract system}. Since $O_k^\top O_k$ is a positive semidefinite matrix, it can be inverted by adding a damping term. The direction is given as
\begin{align}\label{eq: new system for bk}
     d_k = (O_k^\top  O_k)^\dagger O_k^\top  b_k .
\end{align}
From the singular value decomposition, the update \eqref{eq: new system for bk} is also equivalent to 
\begin{align}\label{eq: solution of d_2}
    d_k = O_k^\top  (O_k O_k^\top )^{-1} b_k. 
\end{align}

However, as the norm of the solution in \eqref{eq: new system for bk} increases with the number of samples, directly selection leads to significant errors, especially for small sample sizes. To address this issue, we can incorporate momentum by projecting the historical solution onto the solution set. The parameter update direction $d_k$ is computed iteratively as:
\begin{equation}\label{eq: solution of d_3}
  d_k=  O_k^\top (O_kO_k^\top )^{-1}b_k+\eta (I-O_k^\top (O_kO_k^\top )^{-1}O_k)d_{k-1},
\end{equation}
where $0<\eta<1$ is the decay rate, and $I-O_k^\top (OO_k^\top )^{-1}O_k$ is the projection to the null space of $O_k$.  By projecting the momentum onto the subspace, we can effectively reduce the error in estimating the true solution.

% However, the calculation of $(O_k^\top  O_k)^\dagger$ requires the inversion of a large matrix. We address efficient approximation techniques for this computation in Section \ref{sec : approx}. 

\subsection{Handling the Unavailability of $b_k$}\label{sec-b-unavailable}
In some cases, the target function \(L\) involves high-order derivatives. Computing \(\frac{\delta L}{\delta \rho_{\theta_k}}(x_k^j)\) of \(\Phi_k(x_k^j) \) in \(u_k(x_k^j)\) requires these derivatives, which are often challenging to evaluate directly. Hence, the update \eqref{eq: solution of d_2} and \eqref{eq: solution of d_3} can not be used here since we can not get the value of $b_k$. To address the unavailability of \(b_k\), we leverage an indirect approach.  Although \(u_k(x)\) cannot be computed directly,  it can be approximated through a composition involving another quantity, \(\mathcal{S}(\rho_{\theta_k}(x))\), enabling its subsequent use in computations. 
% Specifically, we consider the following integral from a continuous perspective
% \begin{equation}
%     \int  \partial_{\theta }\mathcal{S}(\rho_{\theta_t})(x)u_t(x)dx,
% \end{equation}
% and assume it can be reformulated into a simplified, more tractable form. 
We define 
\begin{equation}     
S_k=\left[\partial_{\theta}\mathcal{S}(\rho_{\theta_k})(x_k^1),...,\partial_{\theta}\mathcal{S}(\rho_{\theta_k})(x_k^n)\right]^\top\in\mathbb{R}^{n \times p}
       \nonumber
\end{equation}
and assume it has full row rank. Consequently, for any $d_k\in \mathbb{R}^{p \times 1}$, \eqref{eq: general abstract system} holds if and only if the following equation holds
\begin{equation}\label{eq: equiv system with s}
    \frac{1}{n}S_k^\top O_k d_k=\frac{1}{n}S_k^\top b_k.
\end{equation}
This equivalence enables us to reformulate the problem and solve \eqref{eq: general abstract system}. However, the non-symmetric form of the precondition matrix on the left-hand side may lead to numerical instability. To address this, we assume the existence of an equivalent transformation that reformulates the original integral into a semi-definite form. Assume our samples $\{x_k^j\}$ are generated from the distribution $q_k(x)$ (see Table \ref{table: a and ut} for more examples). Therefore, the summation over samples $S_k^\top O_k$ can be viewed as an expectation as 
\begin{align} \label{eq: approx_sto}
    \frac{1}{n} S_k^\top O_k\approx& \mathbb{E}_{x\sim q_k(x)} \left[\partial_\theta \mathcal{S}(\rho_{\theta_k})(x)\partial_\theta\mathcal{A}(\rho_{\theta_k})(x)^\top\right].
\end{align}
Assume that through integral by parts, the right hand side of \eqref{eq: approx_sto} can be reformulated into a semi-definite form using a function-valued map $\mathcal{B}(\rho_{\theta_k})$ as: 
\begin{align*}
    \mathbb{E}_{x\sim q_k(x)} \left[\partial_\theta \mathcal{S}(\rho_{\theta_k})(x)\partial_\theta\mathcal{A}(\rho_{\theta_k})(x)^\top\right] = \mathbb{E}_{x\sim q_k(x)} \left[\partial_\theta \mathcal{B}(\rho_{\theta_k})(x)\partial_\theta\mathcal{B}(\rho_{\theta_k})(x)^\top\right].
\end{align*}
Denote $B_k = \left[\partial_{\theta}\mathcal{B}(\rho_{\theta_k})(x_k^1),...,\partial_{\theta}\mathcal{B}(\rho_{\theta_k})(x_k^n)\right]^\top$ as the discretization of $\partial_{\theta}\mathcal{B}(\rho_{\theta_k})$ on samples. It yields the following approximation 
\begin{align*}
    \frac{1}{n} B_k^\top B_k \approx \mathbb{E}_{x\sim q_k(x)} \left[\partial_\theta \mathcal{B}(\rho_{\theta_k})(x)\partial_\theta\mathcal{B}(\rho_{\theta_k})(x)^\top\right].
\end{align*}
Consequently, we can use $B_k^\top B_k$ to approximate $S_k^\top O_k$ in \eqref{eq: equiv system with s}.

A similar approach is also used to estimate the right hand side of \eqref{eq: equiv system with s} as $v_k$. Further details of practical methods for certain metric can be found in Table \ref{table: a and ut} and the following sections. 
Hence, the remaining task is to compute the update direction
\begin{align}\label{eq: modified update with B}
    d_k=\left(\frac{1}{n}B_k^\top B_k\right)^\dagger v_k,
\end{align}
where various methods, such as KFAC, can be employed to efficiently handle the matrix-inverse-vector product.
{\remark In most cases, we take $\mathcal{S} = \mathcal{A}$ directly with $B_k = O_k$. The update \eqref{eq: modified update with B} can be seen as a least squares solution of the original problem \eqref{eq: general abstract system}.}

\subsection{Approximation Algorithms}\label{sec : approx}
Directly calculating $(O_k^\top O_k)^\dagger \in \mathbb{R}^{p \times p}$ in \eqref{eq: new system for bk} and $\left(B_k^\top B_k\right)^\dagger$ in \eqref{eq: modified update with B} is complicated due to the huge matrix size.  In this subsection, an approximation technique using the Kronecker decomposition is introduced to reduce the calculation overhead. We focus on the $k$-th iteration, and omit the subscript for simplicity.
   
 Consider the state variable $\rho_\theta \in \mathcal{M}$ is parameterized by a feedforward network with $K$ layers with a collection of the weight matrices for each layer $\theta = \left( \vecop(W_1)^\top,\cdots,\vecop(W_K)^\top \right)^\top \in \mathbb{R}^p$. We omit the bias term since it can be incorporated into the weight matrix. Let $\theta^{(l)} = \vecop(W_l) \in \mathbb{R}^{n_l n_{l-1}}$, $b_i=\partial_{\theta^{(l)}} \mathcal{A}(\rho_{\theta^{(l)}})(x^i)$ and $G_l =\frac{1}{n} \sum_{i=1}^n b_ib_i^\top$. The preconditioner  in \eqref{eq: modified update with B}, \eqref{eq: modified update with B Hs}, and \eqref{eq: modified update with B W2} can be approximated as a block diagonal matrix  
    $\text{Diag}\{G_1, G_2, \cdots, G_K \}$.
 For certain input $x^i$, we denote the input of the $l$-th layer by $a_{l-1}^i$, and $s_l^i = W_l a_{l-1}^i$. It follows that 
\begin{align}
    \begin{aligned}
     \partial_{\theta^{(l)}} \mathcal{A}(\rho_{\theta^{(l)}})(x^i) & =\vecop \left(\frac{\partial  \mathcal{A}(\rho_{\theta^{(l)}})(x^i)}{\partial s_l^i}{a_{l-1}^i}^\top\right) =a_{l-1}^i \otimes \frac{\partial   \mathcal{A}(\rho_{\theta^{(l)}})(x^i)}{\partial s_l^i}.
    \end{aligned}
\end{align}
Assuming that $a_{l-1}$ and $\frac{\partial   \mathcal{A}(\rho_{\theta^{(l)}})}{\partial s_l}$ are independent with respect to the sample distribution, we can approximate $G_l$ by
\begin{align} \label{eq: kfac-approx}
    \begin{aligned}
         G_l 
         %&= \sum_{i=1} ^n \partial_{\theta^{(l)}} \mathcal{A}(\rho_{\theta^{(l)}})(x^i) ( \partial_{\theta^{(l)}} \mathcal{A}(\rho_{\theta^{(l)}})(x^i))^\top  \\
        &=\frac{1}{n}\sum_{i=1} ^n ( a_{l-1}^i (a_{l-1}^i)^\top ) \otimes  \left( \frac{\partial   \mathcal{A}(\rho_{\theta^{(l)}})(x^i)}{\partial s_l^i} ^\top \frac{\partial   \mathcal{A}(\rho_{\theta^{(l)}})(x^i)}{\partial s_l^i}  \right) \\
        & \approx   \underbrace{\left(\frac{1}{n}\sum_{i=1} ^n a_{l-1}^i (a_{l-1}^i)^\top \right)}_{A_{l-1}} \otimes  \underbrace{\left( \frac{1}{n}\sum_{i=1} ^n \frac{\partial   \mathcal{A}(\rho_{\theta^{(l)}})(x^i)}{\partial s_l^i} ^\top \frac{\partial   \mathcal{A}(\rho_{\theta^{(l)}})(x^i)}{\partial s_l^i}  \right)}_{S_l}.
    \end{aligned}
\end{align}
It yields $G_l^{\dagger} \approx (A_{l-1} \otimes S_l)^{\dagger}=A_{l-1}^{\dagger} \otimes S_l^{\dagger}$.  This approximation efficiently reduces the computation overhead for the matrix inverse from $\mathcal{O}((n_ln_{l-1})^3)$ to $\mathcal{O}(n_l^3 + n_{l-1}^3)$.

\section{Accelerated Natural Gradient Methods with Specific Metrics}
\label{sec: Examples}
\subsection{$L^2$ Metric}
The $L^2$ metric mainly serves for PDE-based optimization problems, where the functional mapping $\mathcal{A}(\rho)$ and metric tensor $\mathcal{G}(\rho)$ are both the identity. In this case, we focus on working with fixed sampled points, thus omitting the subscripts for the samples. We have
\begin{equation}
b_k=\left[\left(\Psi_k-\beta_k\frac{\delta L}{\delta \rho_{\theta_k}}\right)(x^1),...,\left(\Psi_k-\beta_k\frac{\delta L}{\delta \rho_{\theta_k}}\right)(x^n)\right]^\top. \nonumber
\end{equation}
From \eqref{ode : L2} and \eqref{eq: Psi general update}, the update rule for each $\Psi_k(x^j)$ is given as
\begin{align}\label{eq : update discrete Psi}
    \Psi_k(x^j) =  \mu_k\Psi_{k-1}(x^j) - h_{k} \left( \gamma_k - \dot{\beta}_k - \alpha_k \beta_k \right) \frac{\delta L}{\delta \rho_{\theta_k}}(x^j) .
\end{align}
In this case, $b_k$ is not available since calculating $\frac{\delta L}{\delta \rho_{\theta_k}}(x^j)$ involves high order derivatives, which are difficult to obtain. Hence, we apply \eqref{eq: modified update with B} with $B_k =O_k$ to get $d_k$ for $L^2$ metric. For $v_k = O_k^\top  b_k$ in \eqref{eq: modified update with B}, the following equation holds
\begin{align}\label{eq: Otb}
   v_k = O_k^\top  b_k = \underbrace{\sum_{j=1}^n \partial_\theta \rho_{\theta_k}(x^j) \cdot \Psi_k(x^j)}_{(A)}  - \beta_k  \underbrace{\sum_{j=1}^n\partial_\theta \rho_{\theta_k}(x^j) \cdot \frac{\delta L}{\delta \rho_{\theta_k}} (x^j)}_{(B)}. 
\end{align}
We cannot obtain the (A) and (B) parts in equation \eqref{eq: Otb} through direct multiplication. However, notice that by the chain rule, it holds
\begin{align}\label{eq: chain rule relation}
    \partial_\theta L(\rho_{\theta_k}) = \int \partial_\theta \rho_{\theta_k}(x) \frac{\delta L}{\delta \rho_{\theta_k}}(x)dx \approx \frac{1}{n}\sum_{j=1}^n \partial_\theta \rho_{\theta_k}(x^j) \cdot \frac{\delta L}{\delta \rho_{\theta_k}} (x^j).
\end{align}
The approximation arises from the correlation between the continuous integral in the inner product and the discrete summation over the samples. From \eqref{eq: chain rule relation}, we directly use $\partial_\theta L(\rho_{\theta_k})$ to calculate (B) in \eqref{eq: Otb}. The calculation of $\partial_\theta L(\rho_{\theta_k})$ can be done by automatic differentiation and is easy to get. For (A) in \eqref{eq: Otb}, we introduce an iterative approximation technique to establish an iterative update rule for $v_k$ from $v_{k-1}$ and the gradient towards $\theta$ of the target function $L$. From \eqref{eq : update discrete Psi} and \eqref{eq: Otb}, it yields
\begin{align}\label{eq: calc in w}
    \begin{aligned}
        v_k[i] 
        % =& \sum_{j=1}^n \partial_{\theta[i]} \rho_{\theta_k}(x^j)  \left( \Psi_k(x^j) - \beta_k  \frac{\delta L}{\delta \rho_{\theta_k}}(x^j)   \right) \\
         =& \sum_{j=1}^n   \partial_{\theta[i]}\rho_{\theta_k}(x^j) \left(\mu_k  \Psi_{k-1}(x^j) - \left(\mu_k\beta_k + h_k (\gamma_k - \dot{\beta}_k) \right)\frac{\delta L}{\delta \rho_{\theta_k}} (x^j) \right).
    \end{aligned}
\end{align}
For the last line in \eqref{eq: calc in w}, we can only approximate $\sum_{j=1}^n \partial_{\theta[i]}\rho_{\theta_{k-1}}(x^j) \Psi_{k-1}(x^j) $ from the last $v_{k-1}$. Assuming that the discrete step size is small enough on the manifold, $\rho_{\theta_k}$ and $\rho_{\theta_{k-1}}$ are not far away on the tangent space. Hence, we can approximate $\partial_{\theta_i} \rho_{\theta_{k-1}}$ with $\partial_{\theta_{[i]}} \rho_{\theta_k}$. We employ the following approximation
\begin{align*}
    \sum_{j=1}^n   \partial_{\theta[i]}\rho_{\theta_k}(x^j) \Psi_{k-1}(x^j) \approx \sum_{j=1}^n   \partial_{\theta[i]}\rho_{\theta_{k-1}}(x^j) \Psi_{k-1}(x^j).
\end{align*}
Further we can get the following equation:
\begin{align*}
    \sum_{j=1}^n   \partial_{\theta[i]}\rho_{\theta_{k-1}}(x^j) \Psi_{k-1}(x^j)=v_{k-1}[i] + \sum_{j=1}^n  \beta_{k-1} \partial_{\theta[i]}\rho_{\theta_{k-1}}(x^j) \frac{\delta L}{\delta \rho_{\theta_{k-1}}}(x^j) .
\end{align*}
From this approximation, taking \eqref{eq: chain rule relation} into consideration, the following update rule holds for the vector ${w}_k$ to approximate the original $v_k$:
\begin{align}
     w_k = \mu_k \left( w_{k-1} + n\beta_{k-1} \nabla_{\theta} L(\rho_{\theta_{k-1}})\right)  - n\left(\mu_k\beta_k + h_k (\gamma_k - \dot{\beta}_k) \right) \nabla_\theta L(\rho_{\theta_k}) .
    \label{eq : L2 mom new}
\end{align}
Finally, we can derive the accelerated $L^2$ gradient descent method as Algorithm \ref{alg : L2_1}.
\begin{algorithm}[htbp]
\caption{Accelerated $L^2$ Natural Gradient Descent}
\textbf{Input:} Initial parameters $\theta_0$, step sizes $h_k$, decay rates $\{\alpha_k\}, \{\beta_k\}, \{\gamma_k\}$\\
\textbf{Output:} Updated parameters $\theta_{T}$
\begin{algorithmic}[1]
\FOR{$k = 1, \dots, T-1$}
    \STATE Update $w_k$ according to \eqref{eq : L2 mom new}.
    \STATE Calculate $(O_k^\top O_k)^\dagger$ through samples.
    \STATE Calculate update direction $d_k = (O_k^\top O_k)^\dagger w_k $.
    \STATE Update model parameters $\theta_{k+1} = \theta_k + h_k d_k$.
\ENDFOR
\end{algorithmic}
\label{alg : L2_1}
\end{algorithm}

\subsection{$H^s$ Metric}
The $H^s$ metric is a generalization of the basic $L^2$ metric. According to Table \ref{table: a and ut}, the quantities $O_k$ and $b_k$ in \eqref{eq: general abstract system} are given as: 
\begin{equation}
    \begin{aligned}
       O_k&=[\partial_{\theta}  \rho_{\theta_k}(x^1),...,\partial_{\theta}\rho_{\theta_k}(x^n)]^\top ,   \\ b_k&=\left[ ((\mathbf{D}^s)^* \mathbf{D}^s)^{-1}\left(\Psi_k-\beta_k\frac{\delta L}{\delta \rho_{\theta_k}}\right)(x^1),..., ((\mathbf{D}^s)^* \mathbf{D}^s)^{-1}\left(\Psi_k-\beta_k\frac{\delta L}{\delta \rho_{\theta_k}}\right)(x^n)\right]^\top,
       \nonumber
    \end{aligned}
\end{equation}
where $\mathbf{D}^s \sigma$ is a vector consisting of all derivatives of $\sigma$ up to order $s$. In this case, the term $b_k$ is computationally intractable due to the existence of high-order (positive or negative) derivatives. To overcome this difficulty, we follow the approach outlined in Section \ref{sec-b-unavailable} by selecting $\mathcal{S}(\rho)$ as given in Table \ref{table: a and ut}. 
The treatment of the equivalent system \eqref{eq: equiv system with s} depends on whether $s$ is positive or negative.
\paragraph{Case 1: $s>0$}
We first introduce the following definitions: 
\begin{equation}
    \begin{aligned}
       S_k&=[\partial_{\theta}  ((\mathbf{D}^s)^* \mathbf{D}^s)\rho_{\theta_k}(x^1),...,\partial_{\theta}((\mathbf{D}^s)^* \mathbf{D}^s)\rho_{\theta_k}(x^n)]^\top \in \mathbb{R}^{n\times p},   \\ B_k^{\mathcal{D}}&=[\partial_{\theta}  \mathcal{D}\rho_{\theta_k}(x^1),...,\partial_{\theta}\mathcal{D}\rho_{\theta_k}(x^n)]^\top \in \mathbb{R}^{n\times p},\quad \mathcal{D}\in \mathbf{D}^s,
       \\ z_k&=\left[  \left(\Psi_k-\beta_k\frac{\delta L}{\delta \rho_{\theta_k}}\right)(x^1),..., \left(\Psi_k-\beta_k\frac{\delta L}{\delta \rho_{\theta_k}}\right)(x^n)\right]^\top\in \mathbb{R}^{p\times 1}.
       \nonumber
    \end{aligned}
\end{equation}
Integration by parts then shows that:
\begin{equation}\label{Hs: SO_prod}
    \begin{aligned}
       \frac{1}{n}S_k^\top O_k\approx& \int  \partial_{\theta}((\mathbf{D}^s)^* \mathbf{D}^s) \rho_{\theta_k}(x) \partial_{\theta}\rho_{\theta_k}(x)^\top dx
       \\=&\int  \partial_{\theta}\mathbf{D}^s \rho_{\theta_k}(x) \partial_{\theta}\mathbf{D}^s\rho_{\theta_k}(x)^\top dx 
       \approx \frac{1}{n}\sum_{\mathcal{D}\in \mathbf{D}^s} (B_k^{\mathcal{D}})^\top B_k^{\mathcal{D}}.
    \end{aligned}
\end{equation}
It also holds that
\begin{equation}\label{Hs: SO_prod_1}
    \begin{aligned}
       \frac{1}{n}S_k^\top b_k \approx& \int  \partial_{\theta}((\mathbf{D}^s)^* \mathbf{D}^s) \rho_{\theta_k}(x) ((\mathbf{D}^s)^* \mathbf{D}^s)^{-1}\left(\Psi_k-\beta_k\frac{\delta L}{\delta \rho_{\theta_k}}\right)(x) dx\\
     =& \int  \partial_{\theta}\rho_{\theta_k}(x)\left(\Psi_k-\beta_k\frac{\delta L}{\delta \rho_{\theta_k}}\right)(x) dx
     \approx \frac{1}{n}O_k^\top z_k.
    \end{aligned}
\end{equation}
The approximation is given through the connection between integral of functions and the summation over samples. Thus, the update direction $d_k$ is given by: 
\begin{align}\label{eq: modified update with B Hs}
    d_k = \left(\frac{1}{n}\sum_{\mathcal{D}\in \mathbf{D}^s} (B_k^{\mathcal{D}})^\top B_k^{\mathcal{D}}\right)^{\dagger} \left(\frac{1}{n}O_k^\top z_k\right).
\end{align}
It is worth noting that the term $O_k^\top z_k$ coincides with $v_k$ from the $L^2$ case. Hence, its update can be approximated by that of $w_k$. 
% The required matrix inversion is handled using appropriate approximation methods. 

\paragraph{Case 2: $s<0$}

Similar to \eqref{Hs: SO_prod} and \eqref{Hs: SO_prod_1}, we employ integration by parts to redistribute the high-order derivatives to other terms, thereby alleviating the computational complexity of handling $b_k$.  To be specific, we take $S_k=O_k$ and 
\begin{equation}
    \begin{aligned}
       B_k^{\mathcal{D}}&=[\partial_{\theta}  \mathcal{D}\rho_{\theta_k}(x^1),...,\partial_{\theta}\mathcal{D}\rho_{\theta_k}(x^n)]^\top \in \mathbb{R}^{n\times p},\quad \mathcal{D}\in \mathbf{D}^{|s|},
       \\ z_k^{\mathcal{D}}&=\left[  \mathcal{D}\left(\Psi_k-\beta_k\frac{\delta L}{\delta \rho_{\theta_k}}\right)(x^1),..., \mathcal{D}\left(\Psi_k-\beta_k\frac{\delta L}{\delta \rho_{\theta_k}}\right)(x^n)\right]^\top\in \mathbb{R}^{p\times 1},\quad \mathcal{D}\in \mathbf{D}^{|s|}.
       \nonumber
    \end{aligned}
\end{equation}
It yields the approximation:
\begin{equation}\label{Hs_neg: Sb_prod}
    \begin{aligned}
       \frac{1}{n}S_k^\top b_k \approx& \int  \partial_{\theta} \rho_{\theta_k}(x) ((\mathbf{D}^s)^* \mathbf{D}^s)\left(\Psi_k-\beta_k\frac{\delta L}{\delta \rho_{\theta_k}}\right)(x) dx\\
     =& \int  \partial_{\theta}\mathbf{D}^s\rho_{\theta_k}(x)\mathbf{D}^s\left(\Psi_k-\beta_k\frac{\delta L}{\delta \rho_{\theta_k}}\right)(x) dx
     \approx  \frac{1}{n}\sum_{\mathcal{D}\in \mathbf{D}^s} (B_k^{\mathcal{D}})^\top z_k^{\mathcal{D}}.
    \end{aligned}
\end{equation}
Furthermore, analogous to \eqref{eq: Psi general update}, each $z_k^{\mathcal{D}}$ can be computed via the iterative update:  
\begin{align} 
    \mathcal{D}\Psi_k(x^j) = \mu_k \mathcal{D}\Psi_{k-1}(x^j) +h_{k-1}     \left( \gamma_k - \dot{\beta}_k - \alpha_k \beta_k \right) \mathcal{D}\frac{\delta L}{\delta \rho_{\theta_k}}(x^j) , \quad k\ge 1.
\end{align}

\subsection{Fisher-Rao metric}
\label{sec: discrete FR}
Beyond PDE-based models, we now consider $\rho_\theta$ as a parameterized probability density function.  The $\rho$-trajectory in the Fisher-Rao ARG flow \eqref{Fisher-Rao:flow} can be reformulated as $\partial_t \log \rho_t-\left(\Phi_t-\mathbb{E}_{\rho_t}\left[\Phi_t\right]\right) =0,$
% \begin{equation}\label{log-Fisher-Rao:flow}
% \left\{\begin{array}{l}
% \partial_t \log \rho_t-\left(\Phi_t-\mathbb{E}_{\rho_t}\left[\Phi_t\right]\right) =0, \\
% \partial_t \Psi_t + \alpha_t \Phi_t 
%     + \frac{1}{2} \left(\Phi_t - \mathbb{E}_{\rho_t}\left[\Phi_t\right]\right)\left(\Psi_t  
%     - \mathbb{E}_{\rho_t}\left[\Psi_t\right]\right)-( - \gamma_t +  \dot{\beta}_t)  \frac{\delta L}{\delta \rho_t}=0,
% \end{array}\right. 
% \end{equation}
where $\Phi_t=\Psi_t - \beta_t \frac{\delta L}{\delta \rho_t}$. 
% To adapt this update scheme to the parameter space, we pull back the update for the probability distribution. Specifically, we align the tangent variable induced by the parameter update to the tangent variable of the ODE evolution. This alignment is expressed as  
By defining 
\begin{equation}
    \begin{aligned}
       O_k&=[\partial_{\theta}\log(\rho_{\theta_k})(x_k^1),...,\partial_{\theta}\log(\rho_{\theta_k})(x_k^n)]^\top\in\mathbb{R}^{n\times p} ,   \\b_k&=\left[\overline{\Phi_k}(x_k^1),...,\overline{\Phi_k}(x_k^n)\right]^\top\in\mathbb{R}^{n\times 1},
       \nonumber
    \end{aligned}
\end{equation}
we provide the explicit form of the linear system previously introduced in \eqref{eq: general abstract system}.

Next we discretize  $\Psi$-trajectory for updating the cotangent variable. To estimate $\mathbb{E}_{\rho_k}\left[\Phi_k\right]$ at the $k$-th iteration, it is necessary to update samples $\{x_k^i\}_{i=1}^n \sim \rho_{\theta_k}$ and evaluate $\Phi_k$ at these points. However, a fundamental challenge arises: storing a function state variable $\Phi_k$ for each $x \in \mathbb{R}^d$ is intractable. To address this, we use the solution of the previous linear system, $d_{k-1}$, to estimate centered values of $\Phi_k$ at samples $\{x_k^i\}_{i=1}^n$ in the $k$-iteration as follows:
 \begin{equation}
     \overline{\Phi}_{k-1}(x_k^i)\overset{\triangle}{=}\left\langle \partial_\theta \log \rho_{\theta_{k-1}}(x_k^i), d_{k-1} \right\rangle. 
     \nonumber
 \end{equation} 
 Specifically, we discretize the $\Psi$-trajectory as:
\begin{subequations}
\begin{align}
\frac{\overline{\Psi}_{k}(x_k^i) - \overline{\Psi}_{k-1}(x_k^i)}{h_k} &= - \alpha_k \overline{\Phi}_{k-1}(x_k^i) - \frac{1}{2} \overline{\Phi}_{k-1}(x_k^i)\,\overline{\Psi}_{k-1}(x_k^i)  - (\gamma_k - \dot{\beta}_k) \overline{\frac{\delta L}{\delta \rho_{\theta_k}}}(x_k^i), 
\nonumber
% \\
%  \left\langle \partial_\theta \log \rho_{\theta_k}(x_k^i), d_k \right\rangle &= \overline{\Psi_k(x_k^i)}-\beta_k\overline{\frac{\delta L}{\delta \rho_{\theta_k}}(x_k^i)},\nonumber
\end{align}
\end{subequations}
where $\overline{\Psi}_{k-1}(x_k^i)=\overline{\Phi}_{k-1}(x_k^i)+\beta_{k-1}\overline{\frac{\delta L}{\delta \rho_{\theta_{k-1}}}}(x_k^i)$. Here we use the centralized cotangent variable over samples since centralization does not influence the iteration update. Consequently, the update for the cotangent variable is derived as:
 \begin{equation}\label{Fisher_cot:update}
 \begin{aligned}
 \overline{\Phi}_{k}(x_k^i)=& (1-h_k\alpha_k)\overline{\Phi}_{k-1}(x_k^i)
    - \frac{h_k}{2} \overline{\Phi}_{k-1}(x_k^i)\,\overline{\Psi}_{k-1}(x_k^i) 
    \\&+(\beta_{k}-h_k\dot{\beta}_k-h_k\gamma_k)\overline{\frac{\delta L}{\delta \rho_{\theta_k}}}(x_k^i)-\beta_{k-1}\overline{\frac{\delta L}{\delta \rho_{\theta_{k-1}}}}(x_k^i).
      \end{aligned}
 \end{equation}

Based on the above discussion, we are ready to present the ANGD method in Algorithm \ref{alg:fisher-rao} for the Fisher-Rao metric.
 
\begin{algorithm}[htbp]
\caption{Accelerated Fisher-Rao Natural Gradient Descent}
\textbf{Input:} Initial model  parameters $\theta_0$, step sizes $\{h_k\}$, decay rates $\{\alpha_k\}$, $\{\beta_k\}$, $\{\gamma_k\}$.\\
\textbf{Output:} Updated model  parameters $\theta_{K}$.

\begin{algorithmic}[1]\label{alg:fisher-rao}
\STATE Initialize  $\Phi_{0}=\mathbf{0}\in\mathbb{R}^{n}$.
\STATE Sample  $\{x_0^i\}_{i=1}^n\overset{\text{i.i.d.}}{\sim} \rho_{\theta_0}$.
\FOR{$k = 0, 1, \dots, K-1$}
        \STATE Estimate $\overline{\Phi}_{k-1}(x_k^i)= \left\langle\partial_\theta\log \rho_{\theta_{k-1}} (x_k^i), d_{k-1}\right\rangle$ for $1\le i\le n$.
        \STATE Update the  cotangent variable $b_k[i]\overset{\triangle}{=}\overline{\Phi}_k (x_k^i)$ according to \eqref{Fisher_cot:update} for $1\le i\le n$.
    \STATE Compute the update direction $d_k$ by solving $O_k d_k = b_k$.
    \STATE Update model parameters $
    \theta_{t+1} = \theta_k + h_k d_k.$
            \STATE Update samples $x^i_{k+1}$ based on $x_k^i$ via sampling methods for $1\le i\le n$.
\ENDFOR
\end{algorithmic}
\end{algorithm}

\subsection{\W-2 metric}
The evolution of probability density can equivalently be interpreted as the movement of particles. From the perspective of the continuity equation, \eqref{W2:flow} explicitly defines the velocity of each particle 
$x_t\sim\rho_t$   as $\dot{x}_t=\nabla \Phi_t(x_t)$, where $\Phi_t=\Psi_t-\beta_t \frac{\delta L}{\delta \rho_t}$. Substituting this into the evolution of $\Psi_t$ in \eqref{W2:flow} and taking the spatial gradient, we obtain the particle-wise velocity evolution: 
\begin{equation} 
\begin{aligned}
\frac{d}{dt}\left[\nabla\Psi_t(x_t)\right]
% &=\partial_t \nabla\Psi_t(x_t)+\nabla^2\Psi_t(x_t) \dot{x}_t \\
&= -\alpha_t \nabla\Psi_t+\beta_t\underbrace{\left(\nabla^2\frac{\delta L}{\delta \rho_t}\nabla \Phi_t-\frac{1}{2}\nabla w_t\right)}_{W_{1,t}}+\eta_t\nabla\frac{\delta L}{\delta \rho_t}.
       \nonumber
\end{aligned}
\end{equation}
Note that the evolution of $\log \rho_t$
 given by \eqref{W2:flow} is $\partial_t\log \rho_t+\left\langle\nabla\log \rho_t,\nabla \Phi_t\right\rangle+\Delta\Phi_t=0$, which also  involves second-order derivatives of $\Phi_t$ (or $\Psi_t$).  To capture these dynamics, we further examine the evolution of the spatial Hessian at time-varying samples:
 \begin{equation}
     \begin{aligned}
         \frac{d}{dt}\left[\nabla^2 \Psi_t(x_t)\right]
         % &=\partial_t \nabla^2\Psi_t(x_t)+\nabla^3\Psi_t(x_t) \dot{x}_t \\
         &=-\alpha_t \nabla^2\Psi_t+\beta_t\underbrace{\left(\nabla^3\frac{\delta L}{\delta \rho_t}\nabla \Phi_t-\frac{1}{2}\nabla^2 w_t\right)}_{W_{2,t}}-[\nabla^2 \Phi_t]^2+\eta_t\nabla^2\frac{\delta L}{\delta \rho_t}.
         \nonumber
     \end{aligned}
 \end{equation}
 Thus the full particle-density evolution system is summarized as follows:
 \begin{equation}
%\label{W2:flow}
\left\{
\begin{aligned}
&\partial_t x_t-\nabla\Phi_t(x_t)=0,
    \\&\partial_t \log\rho_t(x) + \left\langle \nabla\Phi_t(x),\nabla\log \rho_t(x)\right\rangle+\Delta\Phi_t(x)  = 0, \hspace{30mm} (\text{fixed }x)\\
    &\frac{d}{dt}\left[ \nabla \Psi_t(x_t)\right] +\alpha_t \nabla\Psi_t(x_t)-\beta_t W_{1,t}-\eta_t\nabla\frac{\delta L}{\delta \rho_t}(x_t)=0,  \\
    &\frac{d}{dt}\left[ \nabla^2 \Psi_t(x_t)\right] +\alpha_t \nabla^2\Psi_t(x_t)-\beta_t W_{2,t}+[\nabla^2\Phi_t(x_t)]^2
   -\eta_t\nabla^2\frac{\delta L}{\delta \rho_t}(x_t)=0,
\end{aligned}
\right.
\nonumber
\end{equation}
where $\Phi_t=\Psi_t+\beta_t\frac{\delta L}{\delta \rho_t}$.
Unlike the Fisher-Rao metric, which requires tracking updates of the cotangent variable at all spatial points in $\mathbb{R}^d$, the current framework for \W-2 metric only needs updates at sample points, making it suitable for practical applications. 
However, a computational challenge arises in evaluating 
 $\nabla w_t$ and $\nabla^2 w_t$. The condition  $\nabla\cdot\left(h_t\nabla\frac{\delta L}{\delta \rho_t}-\rho_t\nabla w_t\right)=0$ does not implies $\rho_t\nabla w_t=h_t\nabla\frac{\delta L}{\delta \rho_t}$, as $\frac{h_t}{\rho_t}\nabla\frac{\delta L}{\delta \rho_t}$ is generally not curl-free and therefore can not guaranteed to be a gradient. To simplify, we approximate $\nabla w_t$ and $\nabla^2 w_t$ as zero, or set $\beta_t\equiv0$ for computational efficiency. 

Now we focus on the case where $\beta_t\equiv0$, which implies  $\Phi_t=\Psi_t$. The linear system \eqref{eq: general abstract system} is specified by taking:
\begin{equation}
    \begin{aligned}
    O_k&=[\partial_{\theta}\log(\rho_{\theta_k})(x_k^1),...,\partial_{\theta}\log(\rho_{\theta_k})(x_k^n)]^\top\in\mathbb{R}^{n\times p},
\\b_k&=\left[\frac{\nabla\cdot(\rho_{\theta_k}\nabla\Phi_k)(x_k^1)}{\rho_{\theta_k}(x_k^1)},...,\frac{\nabla\cdot(\rho_{\theta_k}\nabla\Phi_k)(x_k^n)}{\rho_{\theta_k}(x_k^n)}\right]^\top\in\mathbb{R}^{n\times 1}.
\nonumber
    \end{aligned}
\end{equation}
 Despite this simplification, additional challenges persist. Estimating $b_k$ is computationally intractable due to the difficulty of directly computing or storing the spatial Hessian in most practical scenarios. This aligns with the case discussed in Section \ref{sec-b-unavailable}. To address this, we apply integration by parts, yielding the following identity:
\[\mathbb{E}_{x\sim\rho_{\theta_k}} \left[\frac{\nabla\cdot(\rho_{\theta_k}\nabla\Phi_k)(x)}{\rho_{\theta_k}(x)}\partial_{\theta} \log \rho_{\theta_k}(x)\right]=-\mathbb{E}_{x\sim\rho_{\theta_k}}\left[\partial_{\theta}\left\langle\nabla \log \rho_{\theta_k}(x),\nabla\Phi_k(x)\right\rangle\right],
\] where the right-hand side can be efficiently estimated by sampling and automatic differentiation. Thus we can estimate $d_k$ by 
solving 
\begin{equation}\label{eq: modified update with B W2}
        \left(\frac{1}{n}
    O_k^\top O_k\right)d_k=\frac{1}{n}\sum_{i=1}^n\partial_{\theta}\left\langle\nabla \log \rho_{\theta_k}(x_k^i),\nabla\Phi_{k}(x_k^i)\right\rangle.
    \nonumber
    \end{equation}
This approach requires only the computation of first-order derivatives, eliminating the need for second-order derivatives. 

Synthesizing all the above insights, we can now formulate the ANGD method for Wasserstein-2 metric in Algorithm \ref{alg : W2}. To address discretization errors, we usually  incorporate additional sampling steps (e.g., MCMC) for $x_k^i+h_kV_{k+1}^i$, ensuring that $x_{k+1}^i(1
\le i\le n)$ tend to follow $\rho_{\theta_{k+1}}$.
\begin{algorithm}[htbp]
\caption{Accelerated \W-2 Natural Gradient Descent}
\textbf{Input:} Initial  parameters $\theta_0$, step sizes $\{h_k\}$, decay rates  $\{\alpha_k\}$, $\{\beta_k\}$, $\{\gamma_k\}$.\\
\textbf{Output:} Updated model parameters $\theta_{K}$.

\begin{algorithmic}[1]\label{alg : W2}
\STATE Initialize $V_0^i=\mathbf{0}\in \mathbb{R}^d$ for $1\le i\le n$.
\STATE Sample  $\{x_0^i\}_{i=1}^n\overset{\text{i.i.d.}}{\sim} \rho_{\theta_0}$.
\FOR{$k = 0, 1, \dots, K-1$}

 \STATE Update cotangent gradients  $V_{k+1}^i=(1-\alpha_k h_k)V_{k}^i-h_k\nabla\frac{\delta L}{\delta \rho_{\theta_k}}(x_k^i)$ for $1\le i\le n$.

    \STATE Solve the parameter update direction $d_k$  from 
    \begin{equation}\label{ls:alg-W2}
        \left(\frac{1}{n}
    O_k^\top O_k\right) d_k = \frac{1}{n}\sum_{i=1}^n\partial_{\theta}\left\langle\nabla \log \rho_{\theta_k}(x_k^i),V_{k+1}^i\right\rangle.
    \nonumber
    \end{equation}
    \STATE Update model parameters $
    \theta_{k+1} = \theta_k + h_k d_k.$
     \STATE Update samples $x^i_{k+1}$ based on $x_k^i+h_kV_{k+1}^i$ for $1\le i\le n$.
\ENDFOR
\end{algorithmic}
\end{algorithm}

\section{Theoretical Analysis}\label{sec: theory}
In this section, we present a theoretical analysis of the ARG flow \eqref{R_acc:flow}. First, we rigorously establish the equivalence between \eqref{R_acc:flow} and the system composed of coupled ODEs (\ref{ode : accelerate})–(\ref{ode : accelerate_2}) through the proof of Proposition \ref{prop: Riemannian system}. Subsequently, we derive convergence guarantees for the ARG flow \eqref{R_acc:flow} under the geodesic convexity assumption.

\subsection{Proof of Proposition \ref{prop: Riemannian system}}\label{sec: theory-riemannian system}
\begin{proof}
In the following proof, for notational convenience, we assume that any smooth tangent field defined on curve $\rho_t$ possesses a local extension, which does not impact the final conclusion. We begin by computing $\nabla_{\partial_t \rho_t} \partial_t \rho_t$ using the Koszul formula. For any tangent field $h_t$ along $\rho_t$, it holds:
\begin{equation}\label{P1:0}
\begin{aligned}
    g_{\rho_t}\left(\nabla_{\partial_t \rho_t} \partial_t \rho_t,h_t\right)=&\underbrace{\frac{d}{dt}g_{\rho_t}\left(\partial_t \rho_t,h_t\right)}_{(A)}-\frac{1}{2}\underbrace{h_t \circ g_{\rho_t}\left(\partial_t \rho_t,\partial_t \rho_t\right)}_{(B)}
    \\&+g_{\rho_t}\left(\partial_t \rho_t,[h_t,\partial_t \rho_t]\right),
    \end{aligned}
\end{equation}
where $[\cdot,\cdot]$ denotes Lie brackets. Define $\Phi_t=\mathcal{G}(\rho_t)\partial_t \rho_t \in {T}^*_{\rho_t}\mathcal{M}$, the components on the right side of \eqref{P1:0} can be evaluated using calculus rules:
\begin{equation}\label{P1:1}
\begin{aligned}
 (A)= \frac{d}{dt}  \int \Phi_t h_t dx 
 = \int \partial_t\Phi_t h_t dx+  \int \Phi_t \partial_t h_t dx.
 \end{aligned}
\end{equation}
Besides for component $(B)$, it yields
\begin{equation}\label{P1:2}
  \hfill  (B)=\int \frac{\delta \left(g_{\rho_t}\left(\partial_t \rho_t,\partial_t \rho_t\right)\right)}{\delta \rho_t}  h_t dx + 2\int \mathcal{G}(\rho_t)\partial_t \rho_t(x) \int \frac{\delta}{\delta \rho_t}\partial_t\rho_t(x,y)h_t(y)dy dx.
\end{equation}
Then we compute $[h_t,\partial_t \rho_t]$. For any smooth functional $E(\rho)$, it holds
\begin{equation}
\begin{aligned}\relax
[h_t,\partial_t \rho_t] E(\rho_t)&=h_t\circ \frac{d}{d t} E(\rho_t)-\frac{d}{dt} \int \frac{\delta}{\delta \rho_t} E(\rho_t) h_t dx
% \\&=\int \frac{\delta}{\delta \rho_t}\left[\frac{d}{d t} E(\rho_t)\right] h_t dx+\int \frac{\delta E}{\delta \rho_t}(\rho_t)(x) \int \frac{\delta}{\delta \rho_t}\partial_t\rho_t(x,y)h_t(y)dy dx
% \\&\quad-\left[ \int \frac{d}{dt}\frac{\delta}{\delta \rho_t} E(\rho_t) h_t dx+\int \frac{\delta}{\delta \rho_t} E(\rho_t) \partial_t h_t dx\right]
\\&=\int \frac{\delta E}{\delta \rho_t} (x)  \left(\int \frac{\delta}{\delta \rho_t}\partial_t\rho_t(x,y)h_t(y)dy-\partial_t h_t(x)\right) dx,
\end{aligned} \nonumber
\end{equation}
where the last equality holds due to  $\frac{\delta}{\delta \rho_t}\frac{d}{d t} E(\rho_t)=\frac{d}{d t}\frac{\delta}{\delta \rho_t} E(\rho_t)$. 
It implies:
\begin{equation}\label{P1:3}
    [h_t,\partial_t \rho_t](x)=\int \frac{\delta}{\delta \rho_t}\partial_t\rho_t(x,y)h_t(y)dy-\partial_t h_t(x).
\end{equation}
Substituting \eqref{P1:1}, \eqref{P1:2}, and  \eqref{P1:3} into \eqref{P1:0}, it gives: 
% \overset{\mathrm{(a)}}{=}
\begin{equation}\label{P1:4}
\begin{aligned}
    g_{\rho_t}\left(\nabla_{\partial_t \rho_t} \partial_t \rho_t,h_t\right)=\int \partial_t\Phi_t h_t dx-\frac{1}{2}\int \frac{\delta \left(g_{\rho_t}\left(\partial_t \rho_t,\partial_t \rho_t\right)\right)}{\delta \rho_t}  h_t dx.
    \end{aligned}
\end{equation}
Consequently, we have
\begin{equation}\label{2-covariant-d}
   \mathcal{G}(\rho_t) \nabla_{\partial_t \rho_t} \partial_t \rho_t=\partial_t\Phi_t-\frac{1}{2} \frac{\delta }{\delta \rho_t}\left[g_{\rho_t}\left(\partial_t \rho_t,\partial_t \rho_t\right)\right].
\end{equation}

Next we analyze the expression for $\nabla_{\partial_t\rho_t}\grad L$ using the Koszul formula again.  For any smooth tangent field $h_t$ along the curve $\rho_t$, it holds that:
\begin{equation}\label{R-Hess-damping-comp}
\begin{aligned}
    &2g_{\rho_t}(\nabla_{\partial_t\rho_t}\grad L, h_t)
    \\=&\underbrace{\partial_t \rho_t\circ g_{\rho_t}(\grad L, h_t)}_{(C)}+\underbrace{\grad L\circ g_{\rho_t}(\partial_t \rho_t, h_t)}_{(D)}
    -\underbrace{h_t\circ   g_{\rho_t}(\grad L, \partial_t\rho_t)}_{(E)}
    \\&+\underbrace{g_{\rho_t}([\partial_t\rho_t,\grad L],h_t)}_{(F)}+\underbrace{g_{\rho_t}([h_t,\partial_t\rho_t],\grad L)}_{(G)}+\underbrace{g_{\rho_t}([h_t,\grad L],\partial_t\rho_t)}_{(H)}.
    \end{aligned}
\end{equation}
Calculating the terms using the chain rule gives: 
\begin{equation}\label{metric:1-3-5}
\begin{aligned}
    &(C)-(E)+(G)
    \\=&\partial_t \rho_t\circ(h_t\circ\grad L)-h_t\circ(   \partial_t \rho_t\circ\grad L)+[h_t,\partial_t\rho_t]\circ\grad L = 0,
\end{aligned}
\end{equation}
and it holds for $(D)$ that
\begin{equation}\label{metric:2}
\begin{aligned}
(D) =& \grad L\circ\int \partial_t \rho_t \mathcal{G}(\rho_t) h_t 
\\=&\iint \frac{\delta h_t}{\delta \rho_t}(x,y)\grad L(y)dy\cdot \mathcal{G}(\rho_t) \partial_t \rho_t(x) dx    
    \\&+\iint \frac{\delta h_t}{\delta \rho_t}(x,y)\grad L(y)dy\cdot \mathcal{G}(\rho_t) \partial_t \rho_t(x) dx 
    \\&+ \int\partial_t \rho_t \left(\frac{\partial \mathcal{G}}{\partial \rho_t} \grad L\right)h_tdx  \int\mathcal{G}(\rho_t) h_t(x) \int \frac{\delta}{\delta \rho_t}\partial_t\rho_t(x,y)\grad L(y)dy dx.
\end{aligned}
\end{equation}
For $(F)$ and $(H)$, we use a test functional $E(\rho)$:
\begin{equation}
\begin{aligned}
\relax %括号打头不加会弹warning
 [\partial_t\rho_t,\grad L]\circ E
 &=\frac{d}{dt} \int \frac{\delta L}{\delta \rho_t}\grad L-\grad L \circ \partial_t E(\rho_t)
 % \\=&\iint \frac{\delta^2 L}{\delta \rho_t^2}(x,y)\partial_t\rho_t(y)dy\cdot \grad L(x)dx+\int \frac{\delta L}{\delta \rho_t} \partial_t \grad L(\rho_t)
 % \\&\,-\iint \frac{\delta^2 L}{\delta \rho_t^2}(x,y)\grad L(y)dy\partial_t\rho_t(x)dx-\int \frac{\delta L}{\delta \rho_t}(x) \int\frac{\delta }{\delta \rho_t} \partial_t\rho_t(x,y) \grad L(y)dydx
 \\&=\int \frac{\delta L}{\delta \rho_t} \partial_t \grad L(\rho_t)-\int \frac{\delta L}{\delta \rho_t}(x) \int\frac{\delta }{\delta \rho_t} \partial_t\rho_t(x,y) \grad L(y)dydx,
 \nonumber
 \end{aligned}
\end{equation}
where the last equality hold due to the self-adjoint property of $\frac{\delta^2 L}{\delta \rho_t^2}$. This leads to:  
\begin{equation}\label{Lie:4}[\partial_t\rho_t,\grad L](x)=\partial_t \grad L(x)-\int\frac{\delta }{\delta \rho_t} \partial_t\rho_t(x,y) \grad L(y)dy.
\end{equation} 
Similarly, we obtain: 
\begin{equation}\label{Lie:6}[h_t,\grad L](x)=\int\frac{\delta }{\delta \rho_t} \grad L(x,y) h_t(y)dy-\int\frac{\delta }{\delta \rho_t} h_t(x,y) \grad L(y)dy.
\end{equation} 
Substituting \eqref{metric:1-3-5}, \eqref{metric:2}, \eqref{Lie:4} and \eqref{Lie:6} into \eqref{R-Hess-damping-comp} gives
\begin{align*}
    2g_{\rho_t}(\nabla_{\partial_t\rho_t}\grad L, h_t)
    \overset{\eqref{metric:1-3-5}}{=}&(D)+(F)+(H)
    % \\=&\int\mathcal{G}(\rho_t) h_t(x) \int \frac{\delta}{\delta \rho_t}\partial_t\rho_t(x,y)\grad L(y)dy dx + \int\partial_t \rho_t \left(\frac{\partial \mathcal{G}}{\partial \rho_t}\cdot\grad L\right)h_tdx
    % \\&+\iint \frac{\delta h_t}{\delta \rho_t}(x,y)\grad L(y)dy\cdot \mathcal{G}(\rho_t) \partial_t \rho_t(x) dx
    % \\&+\int  \partial_t \grad L(\rho_t)\cdot \mathcal{G}(\rho_t) h_t-\int\mathcal{G}(\rho_t) h_t(x) \int\frac{\delta }{\delta \rho_t} \partial_t\rho_t(x,y) \grad L(y)dy
    % \\&+\iint\frac{\delta }{\delta \rho_t} \grad L(x,y) h_t(y)dy \cdot\mathcal{G}(\rho_t) \partial_t\rho_t(x)dx
    % \\&- \iint\frac{\delta }{\delta \rho_t} h_t(x,y) \grad L(y)dy\cdot
    % \mathcal{G}(\rho_t) \partial_t\rho_t(x)dx
    \\=&\int \left(\partial_t\mathcal{G}(\rho_t)\cdot\grad L+2\mathcal{G}(\rho_t)\partial_t \grad L(\rho_t)\right)h_t
    \\=&\int \left(\partial_t\frac{\delta L}{\delta \rho_t}+\mathcal{G}(\rho_t)\partial_t \grad L(\rho_t)\right)h_t.
    \nonumber
    \end{align*}
Thus, the conclusion follows:
\begin{equation}\label{R-Hess-damping-result}
\mathcal{G}(\rho_t)\nabla_{\partial_t\rho_t}\grad L=\frac{1}{2}\partial_t\frac{\delta L}{\delta \rho_t}+\frac{1}{2}\mathcal{G}(\rho_t)\partial_t \grad L(\rho_t).
\end{equation}
With \eqref{2-covariant-d} and \eqref{R-Hess-damping-result} substituted into \eqref{R_acc:flow}, we arrive at:
\begin{align}
\begin{aligned}
    &\partial_t \Phi_t + \alpha_t \Phi_t - \frac{1}{2} \frac{\delta}{\delta \rho_t} \left( \int \partial_t\rho_t \mathcal{G}(\rho_t) \partial_t\rho_t \, dx \right)
\\&\quad\quad\quad\quad\quad+\frac{\beta_t}{2}\partial_t\frac{\delta L}{\delta \rho_t}+\frac{\beta_t}{2}\underbrace{\mathcal{G}(\rho_t)\partial_t\grad L(\rho_t)}_{(I)} + \gamma_t\frac{\delta L}{\delta \rho_t}=0.
\end{aligned}
\label{eq: system ode}
\end{align}
To compute $
(I)$, the following expansion is applied:
\begin{equation}
    \begin{aligned}\relax
 (I)&=   \mathcal{G}(\rho_t)\partial_t \left(\mathcal{G}(\rho_t)^{-1}\frac{\delta L}{\delta \rho_t}\right)=\mathcal{G}(\rho_t)\partial_t \left(\mathcal{G}(\rho_t)^{-1}\right)\frac{\delta L}{\delta \rho_t}+ \partial_t \frac{\delta L}{\delta \rho_t}
\\& \xlongequal{\partial_t\rho_t=\mathcal{G}(\rho_t)^{-1}\Phi_t}{} \mathcal{G}(\rho_t)\left[\frac{\partial \left(\mathcal{G}(\rho_t)^{-1}\right)}{\partial \rho_t}\cdot\mathcal{G}(\rho_t)^{-1}\Phi_t\right] \frac{\delta L}{\delta \rho_t}+ \partial_t\frac{\delta L}{\delta \rho_t}.
\end{aligned}
 \end{equation}   

We ultimately establish the conclusion by making transformation $\Psi_t=\Phi_t+\beta_t \frac{\delta L}{\delta \rho_t}$, and  employing 
$
    \left.\frac{\delta}{\delta \rho} \left( \int \partial_t\rho_t \mathcal{G}(\rho) \partial_t\rho_t \, dx \right)\right|_{\rho=\rho_t}=-\left.\frac{\delta}{\delta \rho} \left( \int \Phi_t \mathcal{G}(\rho)^{-1} \Phi_t \, dx \right)\right|_{\rho=\rho_t}.
$
This equality is proved in Section A.2 of \cite{wang2022accelerated}.
\end{proof}
\subsection{Analysis of ODE-flow}
In this subsection, we aim to prove the convergence of the  ODE trajectory \eqref{R_acc:flow} with $\alpha_t = {\alpha}/{t}$. We first provide a technical lemma that analyzes two quantities related to the metrics. 
\begin{lemma}\label{lemma: metric}
    Let $T_t = \log_{\rho_t}\rho^*$ denote the exponential map from $\rho_t$ to $\rho^*$. For the $H^s$ ($s\in\mathbb{Z}$) metric, Fisher-Rao  metric, and \W-2 metric, the following inequality holds:
    \begin{equation}\label{TLemma:1}
        g(\partial_t\rho_t , \nabla_{\partial_t\rho_t} T_t + \partial_t\rho_t)\ge0.
    \end{equation}
    For the $H^s$ ($s\in\mathbb{Z}$) metric and Fisher-Rao  metric, we further have: \begin{equation}\label{TLemma:2}
  g( \grad F(\rho_t), \nabla_{\partial_t\rho_t} T_t + \partial_t\rho_t)\ge g( \grad F(\rho_t),  \partial_t\rho_t)-||\grad F(\rho_t)||_g\cdot||\partial_t\rho_t||_g.
    \end{equation}
\end{lemma}
    \begin{proof}
      \textit{$H^s$ metrics.} The proof is straightforward, as the Riemannian manifold is flat. We have $T_t=\rho^*-\rho_t$, and $\nabla_{\partial_t\rho_t}T_t=\partial_tT_t=-\partial_t\rho_t$. This immediately gives the desired conclusion \eqref{TLemma:1} and \eqref{TLemma:2}.

      \textit{Fisher-Rao metric.} According to Proposition 10 in \cite{wang2022accelerated}, the geodesic connecting $\tau_0$ and $\tau_1$ for any $\tau_0,\tau_1\in \mathcal{M}$ is given by 
      \begin{equation}
       \tau_t(x)=\frac{1}{\sin^2H}  \left[\sin (Ht)\sqrt{\tau_1(x)}+\sin (H(1-t))\sqrt{\tau_0(x)}\right]^2, \quad 0\le t \le 1,
       \nonumber
      \end{equation}
      where $H=\cos^{-1}(\int \sqrt{\tau_0(x)\tau_1(x)}dx)\in[0,\frac{\pi}{2})$. This leads to the expression: \begin{equation}
          T_t=\frac{2H_t}{\sin(H_t)}\sqrt{\rho_t\rho^*}-\frac{2H_t\cos(H_t)}{\sin(H_t)}\rho_t,
      \end{equation}
      where $H_t=\cos^{-1}(\int \sqrt{\rho_t(x)\rho^*(x)}dx)\in[0,\frac{\pi}{2})$. We denote $[f]_{\rho_t} = f - \mathbb{E}_{\rho_t} [f]$. It yields:
\begin{equation}
    \begin{aligned}
        \partial_t H_t=&-\frac{1}{2\sin H_t}\int\sqrt{\frac{\rho^*}{\rho_t}}\partial_t\rho_t \; dx=-\frac{1}{2\sin H_t}\int\sqrt{{\rho^*}{\rho_t}}[\Phi_t]_{\rho_t} \; dx,
        \\\partial_t T_t =&
        % &2\partial_t H_t\frac{\sin H_t-H_t \cos H_t}{\sin^2 H_t}\sqrt{\rho_t\rho^*}+\frac{H_t}{\sin H_t}\sqrt{\frac{\rho^*}{\rho_t}}\partial_t\rho_t
        % \\&-2\partial_t H_t\frac{  \cos H_t\sin H_t-H_t}{\sin^2 H_t}\rho_t-\frac{2H_t\cos H_t }{\sin H_t}\partial_t\rho_t \\ &=
        -\frac{\sin H_t-H_t \cos H_t}{\sin^3 H_t}\left(\int\sqrt{{\rho^*}{\rho_t}}[\Phi_t]_{\rho_t}  dx\right)\left[\sqrt{\frac{\rho^*}{\rho_t}}\right]_{\rho_t}\rho_t
        \\&+\frac{H_t}{\sin H_t}\left(\sqrt{\frac{\rho^*}{\rho_t}}[\Phi_t]_{\rho_t}-\int\sqrt{{\rho^*}{\rho_t}}[\Phi_t]_{\rho_t}  dx\right)\rho_t-\frac{2H_t\cos H_t }{\sin H_t}[\Phi_t]_{\rho_t}\rho_t.
    \end{aligned}
\end{equation}
Thus for any $f\in T^*_{\rho_t}\mathcal{M}$, we have: 
\begin{equation}\label{fisher:ge}
  \begin{aligned}
       \int \partial_t T_t [f]_{\rho_t}
       =&-\frac{\sin H_t-H_t \cos H_t}{\sin^3 H_t} \left(\int\left[\sqrt{\frac{\rho^*}{\rho_t}}\right]_{\rho_t}[\Phi_t]_{\rho_t}\rho_t \; dx \right) \int\left[\sqrt{\frac{\rho^*}{\rho_t}}\right]_{\rho_t}[f]_{\rho_t}\rho_t \; dx \\ &+\frac{H_t}{\sin H_t}\int \sqrt{\rho^*\rho_t}[\Phi_t]_{\rho_t}[f]_{\rho_t} \; dx -\frac{2H_t\cos H_t }{\sin H_t} \int [\Phi_t]_{\rho_t}[f]_{\rho_t}\rho_t \; dx.  
  \end{aligned}
\end{equation}
The Cauchy-Schwarz inequality implies that:
\begin{align} \label{fisher: cauchy}
\begin{aligned}
    &\left(\int\left[\sqrt{\frac{\rho^*}{\rho_t}}\right]_{\rho_t}[\Phi_t]_{\rho_t}\rho_t  dx \right)\left( \int\left[\sqrt{\frac{\rho^*}{\rho_t}}\right]_{\rho_t}[f]_{\rho_t}\rho_t dx\right) \\ &  \leq \left(\int \left[\sqrt{\frac{\rho^*}{\rho_t}}\right]_{\rho_t}^2\rho_t dx \right)\left(\int[\Phi_t]_{\rho_t}^2\rho_t  dx \right)^{\frac{1}{2}}\left(\int[f]_{\rho_t}^2\rho_t  dx \right)^{\frac{1}{2}}  \\& =  \sin^2H_t \left(\int[\Phi_t]_{\rho_t}^2\rho_t  dx \right)^{\frac{1}{2}}\left(\int[f]_{\rho_t}^2\rho_t  dx \right)^{\frac{1}{2}}.
\end{aligned}
\end{align}
      % \begin{equation}\label{fisher:ge}
      % \begin{aligned}
      %      \int \partial_t T_t [f]_{\rho_t}
      %      =&-\frac{\sin H_t-H_t \cos H_t}{\sin^3 H_t} \left(\int\left[\sqrt{\frac{\rho^*}{\rho_t}}\right]_{\rho_t}[\Phi_t]_{\rho_t}\rho_t \; dx \right) \int\left[\sqrt{\frac{\rho^*}{\rho_t}}\right]_{\rho_t}[f]_{\rho_t}\rho_t \; dx \\
      %   &+\frac{H_t}{\sin H_t}\int \sqrt{\rho^*\rho_t}[\Phi_t]_{\rho_t}[f]_{\rho_t}-\frac{2H_t\cos H_t }{\sin H_t}\int [\Phi_t]_{\rho_t}[f]_{\rho_t}\rho_t        \\\overset{\footnotesize{\text{Cauchy-Schwarz}}}{\ge}&-\frac{\sin H_t-H_t \cos H_t}{\sin^3 H_t}\left(\int \left[\sqrt{\frac{\rho^*}{\rho_t}}\right]_{\rho_t}^2\rho_t\right)\left(\int[\Phi_t]_{\rho_t}^2\rho_t\right)^{\frac{1}{2}}\left(\int[f]_{\rho_t}^2\rho_t\right)^{\frac{1}{2}}
      %   \\&+\frac{H_t}{\sin H_t}\int \sqrt{\rho^*\rho_t}[\Phi_t]_{\rho_t}[f]_{\rho_t}-\frac{2H_t\cos H_t }{\sin H_t}\int [\Phi_t]_{\rho_t}[f]_{\rho_t}\rho_t
      %   \\=&-\frac{\sin H_t-H_t \cos H_t}{\sin H_t}\left(\int[\Phi_t]_{\rho_t}^2\rho_t\right)^{\frac{1}{2}}\left(\int[f]_{\rho_t}^2\rho_t\right)^{\frac{1}{2}}
      %   \\&+\frac{H_t}{\sin H_t}\int \sqrt{\rho^*\rho_t}[\Phi_t]_{\rho_t}[f]_{\rho_t}-\frac{2H_t\cos H_t }{\sin H_t}\int [\Phi_t]_{\rho_t}[f]_{\rho_t}\rho_t,
      % \end{aligned}
      % \end{equation}
      % where the last equality uses $\int \left[\sqrt{\frac{\rho^*}{\rho_t}}\right]_{\rho_t}^2\rho_t=\sin^2H_t$.
Since $\tan H_t \geq H_t$ for all $H_t \in [0, \frac{\pi}{2})$,  substituting  \eqref{fisher: cauchy} into \eqref{fisher:ge} yields:
\begin{equation}\label{fisher:geq}
  \begin{aligned}
       \int \partial_t T_t [f]_{\rho_t} dx
       =&-\frac{\sin H_t-H_t \cos H_t}{\sin H_t} \left(\int[\Phi_t]_{\rho_t}^2\rho_t  dx \right)^{\frac{1}{2}}\left(\int[f]_{\rho_t}^2\rho_t  dx \right)^{\frac{1}{2}} \\ &+\frac{H_t}{\sin H_t}\int \sqrt{\rho^*\rho_t}[\Phi_t]_{\rho_t}[f]_{\rho_t}  dx -\frac{2H_t\cos H_t }{\sin H_t} \int [\Phi_t]_{\rho_t}[f]_{\rho_t}\rho_t  dx.  
  \end{aligned}
\end{equation}
We now consider the expression in \eqref{TLemma:1}:
\begin{align*}
    &g(\partial_t\rho_t , \nabla_{\partial_t\rho_t} T_t + \partial_t\rho_t)
    \\=&\frac{d}{dt}g(\partial_t\rho_t , T_t )-g(\nabla_{\partial_t\rho_t}\partial_t\rho_t ,  T_t )+g(\partial_t\rho_t , \partial_t\rho_t)
    \\\overset{\eqref{2-covariant-d}}{=}&\int(\partial_t \Phi_t T_t+ [\Phi_t]_{\rho_t} \partial_tT_t)dx-\int \left(\partial_t\Phi_t-\frac{1}{2}\frac{\delta g(\partial_t\rho_t,\partial_t\rho_t)}{\delta \rho_t}\right)T_tdx+g(\partial_t\rho_t , \partial_t\rho_t)
\\\overset{\eqref{fisher:geq}}{\ge}&-\frac{\sin H_t-H_t \cos H_t}{\sin H_t}\int[\Phi_t]_{\rho_t}^2\rho_t dx
    +\frac{H_t}{\sin H_t}\int \sqrt{\rho^*\rho_t}[\Phi_t]_{\rho_t}^2 dx +\int [\Phi_t]_{\rho_t}^2\rho_t dx 
    \\&-\frac{2H_t\cos H_t }{\sin H_t}\int [\Phi_t]_{\rho_t}^2\rho_t dx-\frac{1}{2}\int[\Phi_t]_{\rho_t}^2 \left(\frac{2H_t}{\sin(H_t)}\sqrt{\rho_t\rho^*}-\frac{2H_t\cos(H_t)}{\sin(H_t)}\rho_t\right) dx
    \\=&0.
    \nonumber
\end{align*}
A similar approach is applied to handle \eqref{TLemma:2}, which gives
      \begin{equation}
      \begin{aligned}
        &g(\grad F(\rho_t), \nabla_{\partial_t\rho_t} T_t + \partial_t\rho_t)
        \\=&\frac{d}{dt}g(\grad F(\rho_t) , T_t )-g(\nabla_{\partial_t\rho_t}\grad F(\rho_t) ,  T_t )+g(\grad F(\rho_t), \partial_t\rho_t)
        % \\\overset{\eqref{R-Hess-damping-result}}{=}&\int \left(\partial_t \frac{\delta F}{\delta \rho_t}T_t+ \frac{\delta F}{\delta \rho_t} \partial_tT_t\right)dx-\int \left(\frac{1}{2}\partial_t\frac{\delta F}{\delta \rho_t}+\frac{1}{2}\mathcal{G}(\rho_t)\partial_t\grad F(\rho_t)\right)T_t dx
        % \\&+g(\grad F(\rho_t), \partial_t\rho_t)
        % \\=&\int \frac{\delta F}{\delta \rho_t} \partial_t T_t-\frac{1}{2}\int\left[\frac{\delta F}{\delta \rho_t}\right]_{\rho_t}\frac{\partial_t\rho_t}{\rho_t}
        % T_t+g(\grad F(\rho_t), \partial_t\rho_t)
        \\\overset{\eqref{fisher:geq}}{\ge}&-\frac{\sin H_t-H_t \cos H_t}{\sin H_t}\left(\int[\Phi_t]_{\rho_t}^2\rho_t
        dx \right)^{\frac{1}{2}}\left(\int\left[\frac{\delta F}{\delta \rho_t}\right]_{\rho_t}^2\rho_t dx\right)^{\frac{1}{2}}
        \\&+\frac{H_t}{\sin H_t}\int \sqrt{\rho^*\rho_t}[\Phi_t]_{\rho_t}\left[\frac{\delta F}{\delta \rho_t}\right]_{\rho_t} dx
        -\frac{2H_t\cos H_t }{\sin H_t} \int [\Phi_t]_{\rho_t}\left[\frac{\delta F}{\delta \rho_t}\right]_{\rho_t}\rho_t dx
        \\&-\frac{1}{2}\int\left[\frac{\delta F}{\delta \rho_t}\right]_{\rho_t}[\Phi_t]_{\rho_t}\left(\frac{2H_t\sqrt{\rho_t\rho^*}}{\sin H_t}-\frac{2H_t\cos H_t}{\sin H_t}\rho_t\right) dx +\int [\Phi_t]_{\rho_t}\left[\frac{\delta F}{\delta \rho_t}\right]_{\rho_t}\rho_t  dx
        \\\ge&-\left(\int[\Phi_t]_{\rho_t}^2\rho_t 
        \right)^{\frac{1}{2}}\left(\int\left[\frac{\delta F}{\delta \rho_t} \right]_{\rho_t}^2\rho_t dx 
        \right)^{\frac{1}{2}}+\int [\Phi_t]_{\rho_t}\left[\frac{\delta F}{\delta \rho_t}\right]_{\rho_t}\rho_t dx,
        \nonumber
        \end{aligned}
      \end{equation}
      where the last inequality uses Cauchy-Schwarz inequality. Thus, the conclusion \eqref{TLemma:1} and \eqref{TLemma:2} holds for Fisher-Rao metric.
      
      \textit{\W-2 metric.} Suppose that $P_t$ is the optimal transport mapping from $\rho_t$ to $\rho^*$. Hence, we have ${P_t}_{\#}\rho_t=\rho^*$. By Brenier's Theorem  \cite{villani2009optimal}, there exists a strictly convex function $c_t : \mathbb{R}^d \to \mathbb{R}$ such that $P_t = \nabla c_t$, and consequently, $\nabla P_t = \nabla^2 c_t$ is symmetric. Then according to the continuity equation, the exponential map is expressed as $T_t=-\nabla\cdot(\rho_t(P_t-id)).$
      
      Next, we compute the time derivative of $P_t$. Noting that
      \begin{equation}
    0=\partial_t(P_t^{-1}\circ P_t)=\partial_t(P_t^{-1})\circ P_t+\nabla(P_t^{-1})\cdot \partial_t P_t,
      \end{equation}
      and $I_d=\nabla(P_t^{-1}\circ P_t)=\nabla (P_t^{-1})\cdot \nabla P_t$, we can deduce
      \begin{equation}\label{partialPt}
      \partial_t P_t(x)=-\nabla P_t(x)  u_t(x), 
      \end{equation}
      where $u_t=\partial_t(P_t^{-1})\circ P_t$. Given that $P_t$ is the optimal transport mapping, the distribution of $y_t={P_t}^{-1}(y_0)$ is $\rho_t$ for $y_0\sim \rho^*$. Its velocity is given by $\partial_t y_t = u_t(y_t)$. By the continuity equation, we have
      \begin{equation}\label{vPhi}
          \nabla\cdot(\rho_tu_t)=\nabla\cdot(\rho_t\nabla\Phi_t)=-\partial_t \rho_t.
      \end{equation}
      Based on the preceding discussion, we now proceed with the term
$g(\partial_t\rho_t , \nabla_{\partial_t\rho_t} T_t + \partial_t\rho_t)$ as follows:
\begin{equation}\label{w2:curv}
\begin{aligned}
 g(\partial_t\rho_t , \nabla_{\partial_t\rho_t} T_t + \partial_t\rho_t)
        % \\=&\frac{d}{dt}g(\partial_t\rho_t , T_t )-g(\nabla_{\partial_t\rho_t}\partial_t\rho_t ,  T_t )+g(\partial_t\rho_t , \partial_t\rho_t)
        \overset{\eqref{2-covariant-d}}{=}&\int\Phi_t \partial_tT_t dx+\int \frac{\delta g(\partial_t\rho_t,\partial_t\rho_t)}{\delta \rho_t}\frac{T_t}{2} dx +g(\partial_t\rho_t , \partial_t\rho_t)
         % \\\overset{(a)}{=}&\int \left\langle\nabla\Phi_t,\partial_t\rho_t\cdot(P_t-x)+\rho_t\cdot\partial_t(P_t-x)\right\rangle dx
         % \\&+\frac{1}{2}\left\langle\nabla ||\nabla\Phi_t||^2,P_t-x\right\rangle\rho_t dx+\int||\nabla\Phi_t||^2\rho_t dx 
\\\overset{(b),\eqref{partialPt}}{=}&\int \left\langle\nabla\left\langle\nabla\Phi_t,P_t-x\right\rangle,\nabla\Phi_t\right\rangle\rho_t dx-\left\langle\nabla\Phi_t,\nabla P_t u_t\right\rangle\rho_t dx
         \\&+\frac{1}{2}\left\langle\nabla ||\nabla\Phi_t||^2,P_t-x\right\rangle\rho_t dx+\int||\nabla\Phi||^2\rho_t dx 
\\=&\int\left\langle\nabla\Phi_t,\nabla P_t \nabla\Phi_t\right\rangle\rho_t dx-\int\left\langle\nabla\Phi_t,\nabla P_t u_t\right\rangle\rho_t dx,
        \end{aligned}
\end{equation}
where  $(a), (b)$ use integration by parts. The equations \eqref{partialPt} and \eqref{vPhi} imply that:
\begin{equation}
\int\left\langle\nabla\Phi_t-u_t,\nabla P_t u_t\right\rangle\rho_t  =-\int\left\langle\nabla\Phi_t-u_t,\partial_t \nabla c_t\right\rangle\rho_t =\int\nabla\cdot(\rho_t(\nabla\Phi_t-u_t))\partial_t c_t =0. 
\nonumber
\end{equation}
Substituting it into \eqref{w2:curv} and using the semi-definiteness of $\nabla P_t = \nabla^2 c_t$, we obtain the following inequality:
\begin{equation}
 g(\partial_t\rho_t , \nabla_{\partial_t\rho_t} T_t + \partial_t\rho_t)
=\int\left\langle\nabla\Phi_t-u_t,\nabla P_t (\nabla\Phi_t-u_t)\right\rangle\rho_t dx\ge0.
\nonumber
\end{equation}
 Thus, the proof is complete.
\end{proof}
Before presenting the convergence theorem, we define $w_t = \gamma_t -\dot{\beta}_t - \beta_t/t$ and $\delta_t = t^2 \left(\gamma_t +(\alpha-3)\beta_t/(2t) -\dot{\beta_t}\right)$.
% \begin{align*}
%     w_t = \gamma_t -\dot{\beta}_t - \frac{\beta_t}{t} \quad \text{and} \quad \delta_t = t^2 \left(\gamma_t +\frac{\alpha-3}{2t}\beta_t -\dot{\beta_t}\right).
% \end{align*}
\begin{theorem}
    Assume that the target function $L$ is geodesically convex towards $\rho$ on the manifold $\mathcal{M}$, and $L$ attains its minimum at $\rho^*$. Let $\rho : [t_0, + \infty) \to \mathcal{M}\,(t_0>0)$ be a solution trajectory of \eqref{R_acc:flow}. Suppose that $\alpha_t={\alpha}/{t}$ with $\alpha>1$, and $\gamma_t>0$.
    Then if $w_t>0$ and $\dot{\delta}_t \leq 2tw_t(\alpha-1)$ hold for the $H^s$($s\in\mathbb{Z}$) metric and the Fisher-Rao metric, or $\beta_t\equiv0$ holds for the \W-2 metric, we have 
    \begin{align}\label{thm: converge}
        L(\rho_t) - L(\rho^*) = \mathcal{O}\left(\frac{1}{t^2 w_t} \right) \quad \text{as}  \quad t \to \infty.
    \end{align}
\end{theorem}
\begin{proof}
For notational convenience, we use a dot over a variable to denote its derivative with respect to time. Consider the following Lyapunov function
\begin{align}\label{eqp : Lyapunov}
    E_t = \delta_t \left(L(\rho_t) - L(\rho^*) \right) + \frac{1}{2} g(v_t,v_t) + \frac{\alpha-1}{2}g(\log_{\rho_t} (\rho^*), \log_{\rho_t} (\rho^*)), 
\end{align}
where $v_t = -(\alpha-1) \log_{\rho_t}(\rho^*) + 2t (\dot{\rho}_t + \beta_t \operatorname{grad} L(\rho_t))$. Since $\rho_t$ is the trajectory of the ODE \eqref{R_acc:flow}, we can calculate the derivative of $v_t$ as
\begin{align}\label{eqp : derivative v}
    \begin{aligned}
        \nabla_{\dot{\rho}_t} v_t = (\alpha - 1)(-\nabla_{\dot{\rho}_t} \log_{\rho_t}(\rho^*) - \dot{\rho}_t) -(\alpha-1) \dot{\rho}_t - 2t w_t \grad L(\rho_t).
    \end{aligned}
\end{align}
% Since $\rho_t$ is the trajectory of the ODE, the following equality holds:
% \begin{align*}
%     t \nabla_{\dot{\rho}_t} \dot{\rho}_t + \alpha \dot{\rho}_t + t \beta_t \nabla_{\dot{\rho}_t} \grad L(\rho_t) + t\gamma_t \grad L(\rho_t) = 0 .
% \end{align*}
% Substitute this into the derivative of $v$ \eqref{eqp : derivative v} yields
% \begin{align*}
%     \dot{v}_t = (\alpha - 1)(-\nabla_{\dot{\rho}_t} \log_{\rho_t}(\rho^*) - \dot{\rho}_t) - t w_t \grad L(\rho_t).
% \end{align*}
It gives
\begin{align*}
    \frac{d}{dt}E_t =& \underbrace{\dot{\delta}_t (L(\rho_t)- L(\rho^*)) +2tw_t(\alpha-1) g( \log_{\rho_t}(\rho^*), \grad L(\rho_t))}_{(A)} \\
    & - 4t^2\beta_t\delta_tg(\grad L(\rho_t),\grad L(\rho_t)) -\underbrace{ (\alpha-1)^2 g(\log_{\rho_t}(\rho^*),-\nabla_{\dot{\rho}_t}\log_{\rho_t}(\rho^*) - \dot{\rho}_t )}_{(B)}   \\
    & -2t(\alpha-1)g(\dot{\rho}_t,\dot{\rho}_t)^2 + \underbrace{2t(\alpha-1)g( \dot{\rho}_t + \beta_t \grad L(\rho_t),  -\nabla_{\dot{\rho}_t}\log_{\rho_t}(\rho^*) - \dot{\rho}_t )}_{(C)}.
\end{align*}
Since the target function $L$ is geodesically convex, we have for part (A):
\begin{align*}
    (A) &\leq \left(\dot{\delta}_t - 2tw_t(\alpha-1)\right) (L(\rho_t)- L(\rho^*)) \leq 0.
\end{align*}
For (B), we have the following equation
\begin{align*}
    (B) = -\frac{(\alpha-1)^2}{2} \frac{d}{dt}  \operatorname{dist}(\rho_t,\rho^*)^2 +\frac{(\alpha-1)^2}{2} \frac{d}{dt}  \operatorname{dist}(\rho_t,\rho^*)^2  = 0.
\end{align*}
One more term occurs different from the Euclidean case is (C). By Lemma \ref{lemma: metric}, for $H^s$, Fisher-Rao and \W-2 metrics ($\beta_t=0$), it yields 
\begin{align*}
    \frac{d}{dt}E_t \leq & - 4t^2\beta_t\delta_t\|\grad L(\rho_t)\|_g^2 -2t(\alpha-1)\|\dot{\rho}_t\|_g^2  + 4t(\alpha-1) \beta_t\|\grad L(\rho_t)\|_g \|\dot{\rho}_t\|_g \\
    = & -2t(\alpha-1) (\|\dot{\rho}_t\|_g -\beta_t^2 \|\grad L(\rho_t)\|_g )^2 \\ &- 2t\beta_t(2t\delta_ t  -\beta_t(\alpha-1)) \|\grad L(\rho_t)\|_g^2  \leq 0.
\end{align*}
The convergence rate \eqref{thm: converge} follows from the monotonic decreasing property of $E_t$.
\end{proof}

% \subsection{Analysis of discrete version}
% \begin{itemize}
%     \item \cite{jin2024parameterizedwassersteingradientflow} For $W2$ distance, give an approximation of the information matrix. Analysis for the approximation error. And uniform convergence analysis under PL condition for continuous-time ode. No analysis for discrete algorithm.
% \end{itemize}
% \begin{itemize}
%     \item Idea: Convergence on the manifold + control the projection error.
%     \item  Assumption: strongly convex of $L$ towards $\rho$.
% \end{itemize}

\section{Numerical Experiments}\label{sec: experiments}
In this section, we give some numerical examples to show how machine learning problems can be fitted in form \eqref{problem: form}. Our method exhibits superior numerical performance on these examples.

\subsection{The Burgers' Equation}\label{sec: Burgers}
This subsection addresses the Burgers' equation, which is known for its challenges associated with shock waves and discontinuities. The equation, supplemented with boundary data $h(x)$, is formalized as follows:
\begin{align}
    \begin{aligned}
        &u_t + u u_x - \frac{0.01}{\pi} u_{xx} = 0, \quad x \in [-1, 1], \quad t\in [0, 1],\\
        &u(0,x) = h(x), \quad u(t,-1)=u(t,1) = 0.
    \end{aligned}
\end{align}
Let $\Omega = [-1,1]\times [0,1]$ and define $\partial \Omega_p = \{-1, 1\} \times [0,1] \cup [-1, 1] \times \{0\}$. A neural network $u_\theta$ is employed to approximate the solution with six hidden layers with $(20,50,80,80,50,20)$ neurons. The associated PDE and boundary loss for $u_\theta$ are:
\begin{align}
    L(u_\theta) = \|(u_\theta)_t + u_\theta (u_\theta)_x - \frac{0.01}{\pi} (u_\theta)_{xx} \|^2_{L^2(\Omega)} + \lambda \|u_\theta - g \|^2_{L^2(\partial \Omega)},
    \label{equation: total loss burgers}
\end{align}
where the function $g(x,t) = h(x)$ for $(x, t) \in [-1, 1] \times \{0\}$ and $g(x,t) = 0$ for $(x, t) \in \{-1, 1\} \times [0,1]$ represents the initial and boundary conditions. Taking $u_\theta$ on the  $L^2$ space, the training problem \eqref{equation: total loss burgers} can be fitted into \eqref{problem: form}. 
% The manifold denotes the prior knowledge we have towards the equation to solve. 

In our investigation, we evaluate the efficacy of the ANGD method, comparing against the stochastic gradient descent (SGD) algorithm, Adam \cite{kingma2014Adam} and natural gradient method without acceleration (NGD). The most important difference between the ANGD method and the NGD method is whether or not acceleration is considered on the manifold. We employ a systematic grid search to identify hyper-parameters for several algorithms. For Adam, we vary the initial learning rate among $\{0.001, 0.005, 0.01\}$, the parameters for the momentum terms $\beta_1 \in \{0.9, 0.99\}$ and $\beta_2 \in \{0.99, 0.999\}$, and the weight decay from the set $\{0, 1\text{e-4}, 5\text{e-5}\}$. Similarly, for SGD, ANGD and NGD, the optimal configurations are determined by grid searching across the same ranges for the initial learning rate and weight decay. The ODE parameters $\alpha_k$ and $\beta_k$ are initially set to the optimal values chosen from $\{0.01$, $0.05$, $0.1$, $0.15\}$ and subsequently decay linearly.

We examine two distinct boundary conditions  $h(x) = \sin(\pi x)$ and $h(x) = 1-\cos(2\pi x)$, each presenting varying levels of training difficulty. The efficacy of ANGD is demonstrated in Figures \ref{fig: burgers1} and \ref{fig: burgers2}. We consider the training loss and the testing loss versus iterations. ANGD demonstrates substantial convergence improvements over NGD, while surpassing both Adam and SGD. The lowest loss is also attained by the ANGD method.  As a natural gradient method, ANGD consistently maintains a significantly lower testing loss compared to Adam and SGD after acceleration. 
\begin{figure}[htbp]
    \centering
    \includegraphics[width=0.84\textwidth]{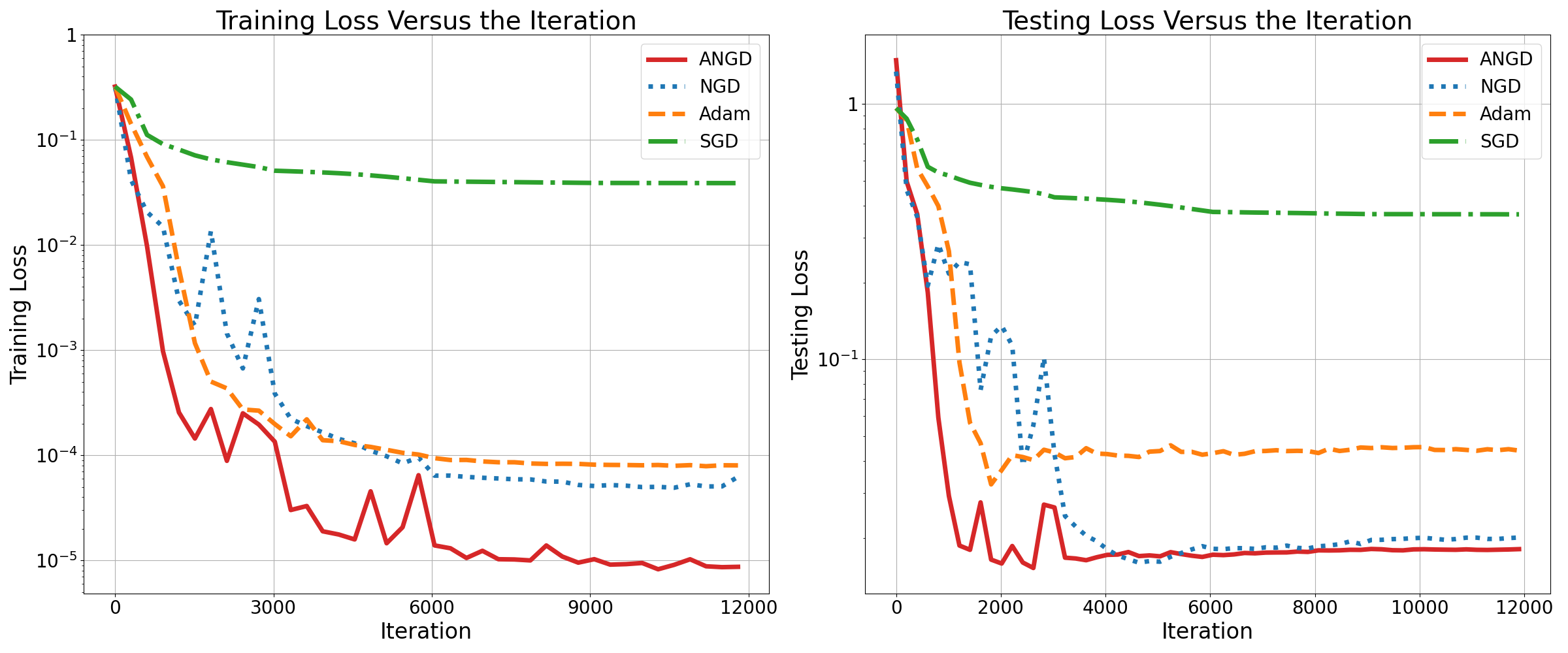}
    \caption{Numerical results for boundary condition $h(x) = \sin(\pi x)$}
    \label{fig: burgers1}
\end{figure}
\begin{figure}[htbp]
    \centering
    \includegraphics[width=0.84\textwidth]{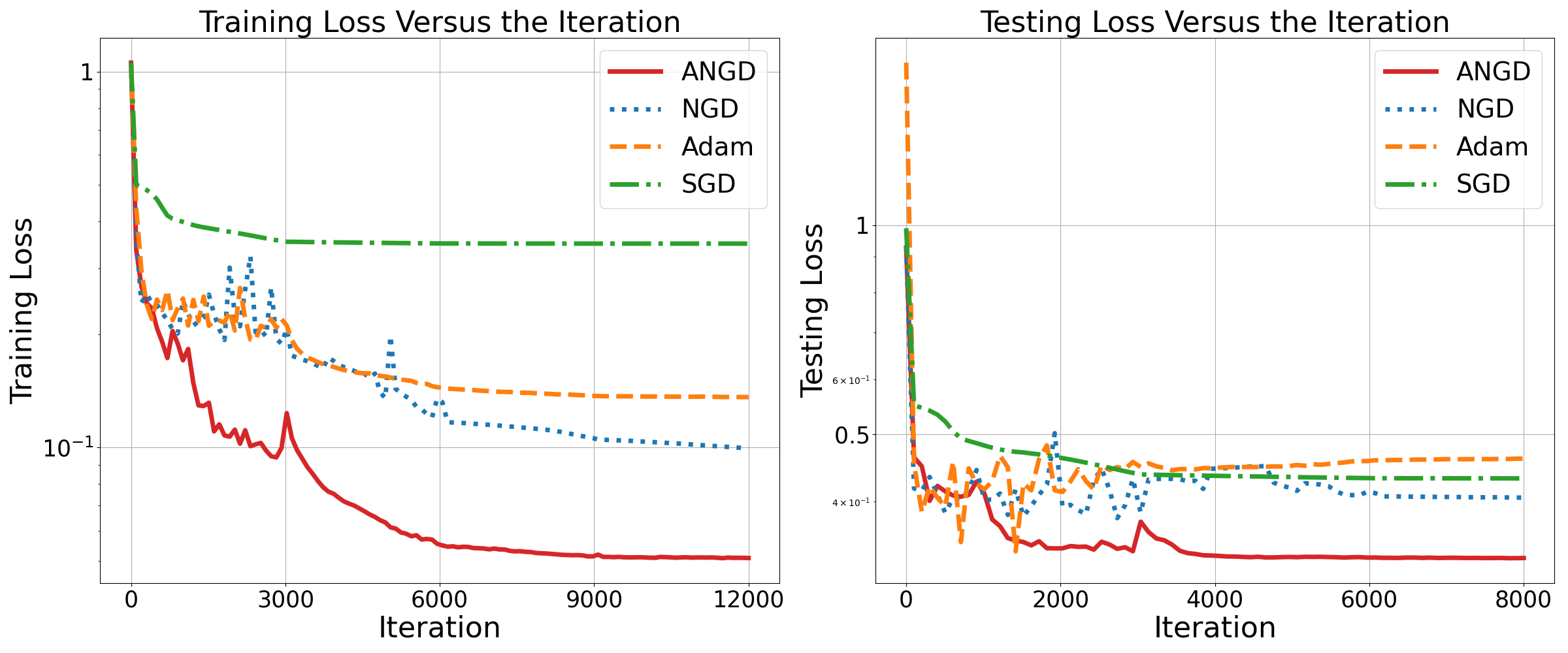}
    \caption{Numerical results for boundary condition $h(x) = 1 - \cos(2\pi x)$}
    \label{fig: burgers2}
\end{figure}

\subsection{The Euler Equations}
In this subsection, we solve the following conservative hyperbolic PDE:
\begin{align}
    \frac{\partial U}{\partial t}  + \nabla \cdot F = 0.
\end{align}
A notable example is the Euler equations, which represent a complex fluid dynamics problem frequently involving discontinuities. For the one-dimensional case of the Euler equations, the vectors $U$ and $F$ are defined as $U=(\rho,\rho u,E)\in\mathbb{R}^3, F=(\rho u,\rho u^2 + p,u(E+p))\in\mathbb{R}^3$. 
% \begin{align*}
%     U=\begin{pmatrix}
%     \rho \\
%     \rho u \\
%     E
%     \end{pmatrix}, \quad
%     F=\begin{pmatrix}
%     \rho u \\
%     \rho u^2 + p \\
%     u(E+p)
%     \end{pmatrix}.
% \end{align*}
In the context of an ideal gas, $\rho$ symbolizes the density, $u$ represents the velocity, $p$ denotes the pressure, and $E= \frac{1}{2}\rho u^2 + \frac{p}{0.4}$ is the total energy. We employ a neural network $g_\theta = (\rho_\theta, u_\theta, p_\theta)$ with input $(x, t)$ designed to simultaneously approximate $\rho$, $u$, and $p$. Through this parametrization, we can get the value of vectors $U$ and $F$ as $U_\theta$ and $F_\theta$. The initial condition is the same as the Sod problem, which has been extensively studied. It is a 1D Riemann problem with the initial constant states in a tube with unit length formulated as
\begin{equation*}
    g(x,0) = (\rho, u, p) = 
    \begin{cases} 
    (1, -2, 0.4) & \text{if } 0 \leq x \leq 0.5, \\
    (1, 2, 0.4) & \text{if } 0.5 < x \leq 1.
    \end{cases}
\end{equation*}
We test the performance of the network at $t=0.2s$. To quantify the performance of our model, the loss function on the area $\Omega = [0, 1] \times [0, 0.2]$ is given as 
\begin{align}\label{eq: euler loss}
    L(\rho_\theta, u_\theta, p_\theta) = \| (U_\theta)_t + \nabla \cdot F_\theta\|^2_{L^2(\Omega)} + \lambda \|g_\theta - g(x,0)\|^2.
\end{align}
Here we consider the manifold  $L^2(\Omega)^{\otimes 3}$,  allowing \eqref{eq: euler loss} to be fitted in the form of \eqref{problem: form}. We evaluate the efficacy of ANGD method using $L^2$ metric, comparing its performance against SGD, Adam, and non-accelerated NGD. The hyper-parameters are set as described in Section \ref{sec: Burgers}. Figure \ref{fig: euler} illustrates the evolution of training and testing loss with respect to iterations. The ANGD method demonstrates a faster convergence rate in terms of training loss compared with the conventional method and the non-accelerated NGD, highlighting its efficiency in optimizing PINNs. Moreover, the ANGD method achieves a substantially lower testing loss, indicating generalization and improved alignment between the network’s predictions and the ground truth.
\begin{figure}[htbp]
\centering
\includegraphics[width=0.84\textwidth]{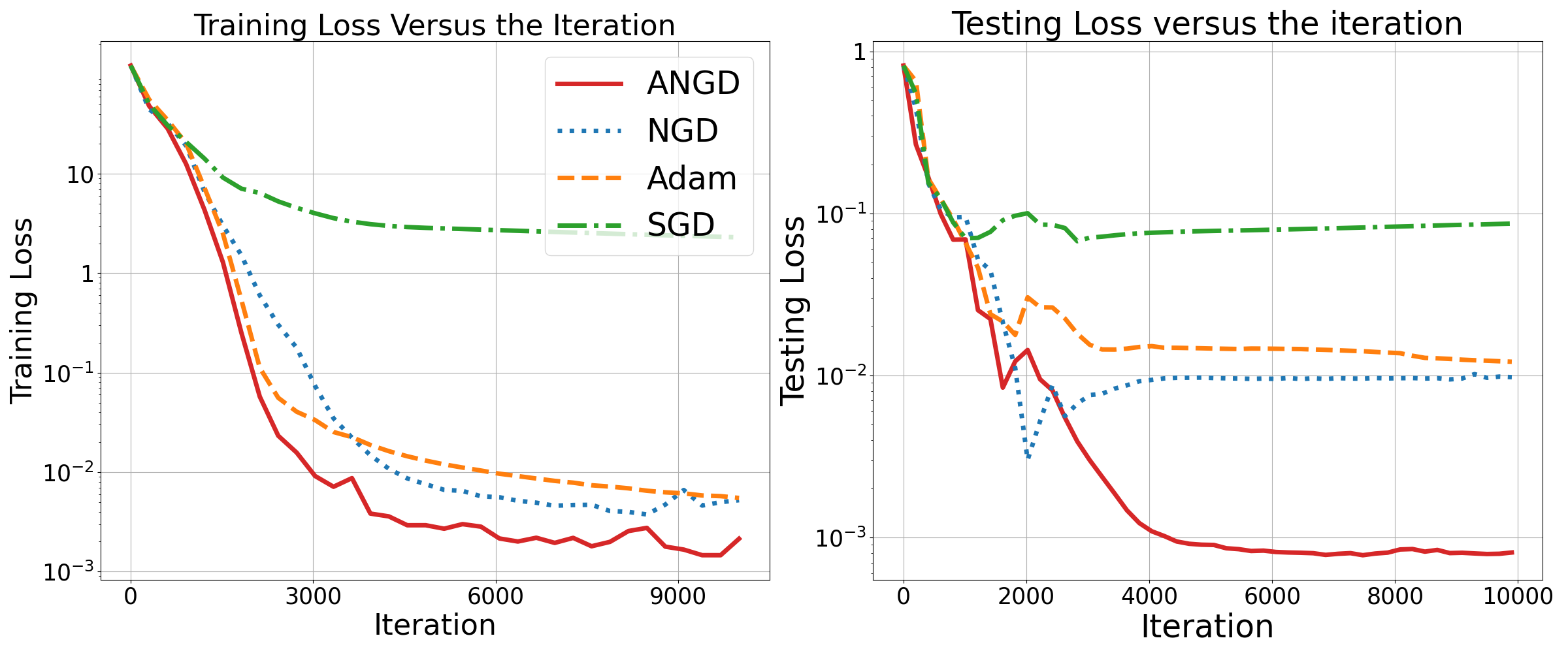}
    \caption{Numerical results for solving the Euler equations}
    \label{fig: euler}
\end{figure}
\subsection{Many-body quantum problem}\label{sec: experiments-vmc}
We consider a many-body  quantum system with $N$ electrons $x=\{x_1,...,x_N\}\in\mathbb{R}^{3N}$. The wavefunction $\psi_\theta : x \to \mathbb{R}$ describing the 
quantum state of the system is typically parameterized using neural networks, such as Ferminet \cite{pfau2020ab}. The goal is to solve for the ground state energy  and wavefunction, which is  formulated as a variational problem:
\begin{equation}\label{VMC}
\min_{\theta}\frac{\int _{x\in \mathbb{R}^{3N}}\psi_\theta(x)(\mathcal{H}\psi_\theta)(x)dx}{\int _{x\in \mathbb{R}^{3N}}\psi_\theta(x)^2dx}=\int_{x\in \mathbb{R}^{3N}} \sqrt{\rho_\theta}(x) (\mathcal{H}\sqrt{\rho_\theta})(x)dx\overset{\triangle}{=}L(\rho_\theta),
\end{equation}
where $\mathcal{H}$ is a  Hamiltonian operator, and $\rho_\theta:=\frac{\psi_\theta^2}{\int_{x\in \mathbb{R}^{3N}}\psi_\theta^2dx}$ is a probability density.  From this, we derive   $\frac{\delta L}{\delta \rho_\theta}=\frac{\mathcal{H}\sqrt{\rho_\theta}}{\sqrt{\rho_\theta}}=\frac{\mathcal{H}\psi_\theta}{\psi_\theta}$ and  $\partial_\theta \log \rho_\theta(x)=2(\partial_\theta\psi_\theta(x)-\mathbb{E}_{\rho_\theta}[\partial_\theta\psi_\theta])$.

Variational Monte Carlo (VMC) methods utilize Markov Chain Monte Carlo (MCMC) sampling to estimate expectations with an unnormalized probability distribution. Using this approach,  we conduct numerical experiments on a small atom (Be), and three molecules ($\text{Li}_2,\text{H}_{10},\text{N}_2$), to verify the acceleration effects of our  proposed ANGD algorithms on the Fisher-Rao metric using projected momentum discretization outlined in  \eqref{eq: solution of d_3}, and the \W-2 metric with  KFAC discretization. Notably, the non-accelerated NGD-Fisher-Rao algorithm corresponds to the SPRING algorithm \cite{Goldshlager2024spring}, while the non-accelerated NGD-\W-2 algorithm is essentially the WQMC algorithm \cite{neklyudov2023wasserstein}, differing primarily in numerical stability techniques.

\begin{figure}[t!]
    \centering
\includegraphics[width=0.84\textwidth,height=0.6\linewidth]{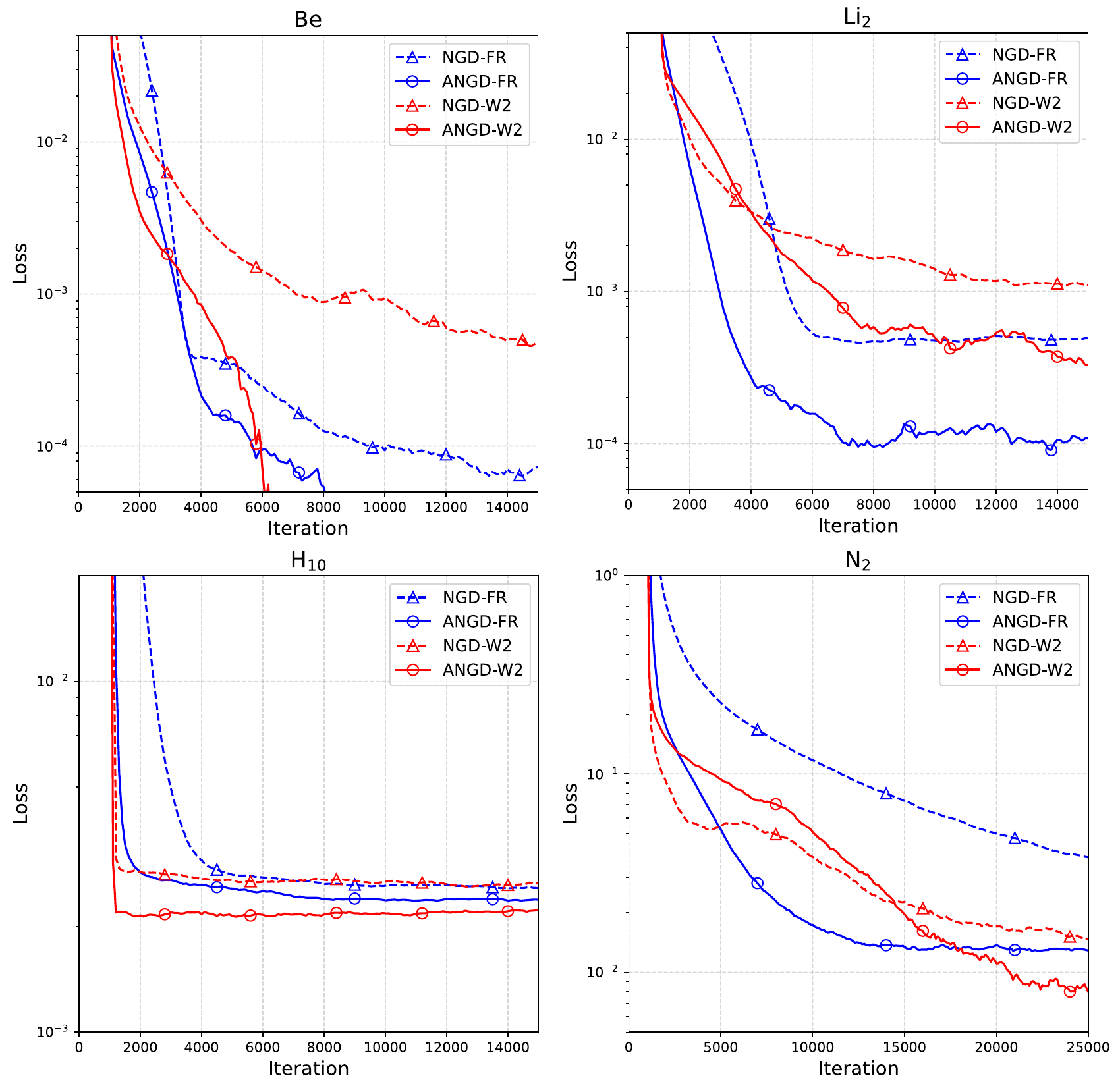}
    \caption{Numerical results of  VMC  on the molecules $\text{Be},\text{Li}_2,\text{H}_{10},\text{N}_2$. We use ``FR" and  ``W2" to denote the Fisher-Rao and \W-2 metrics, respectively, for brevity.}
    \label{fig:VMC}
\end{figure}

The experimental setup is as follows. The sample size $n$ is set to 2000 for the Fisher-Rao metric and 4000 for the \W-2 metric. 
The initial learning rate $h_0$ is searched
in $\{0.001$, $0.005$, $0.01$, $0.05$, $0.1\}$  for the baseline NGD algorithms, and in $\{\sqrt{0.001}$, $\sqrt{0.005}$, $\sqrt{0.01}$, $\sqrt{0.02}$, $\sqrt{0.05}\}$ for the ANGD algorithms. Following the approach in \cite{Goldshlager2024spring}, all algorithms employ a linear decay schedule $h_k=\frac{h_0}{1+\epsilon k}$ with  $\epsilon=5\text{e-5}$ for ANGD and $\epsilon=1\text{e-4}$ for NGD (doubled due to the square root). We also impose linear decay on  $\alpha_k, \beta_k$ with  $\alpha_0\in\{0.1/h_0,$ $0.2/h_0,$ $0.5/h_0\}$, $\beta_0\in$ $\{0.0, 0.01,0.05,0.1,0.15\}$ for ANGD methods, and set $\beta_k\equiv 0$ for ANGD-\W-2. Other hyper-parameters including   clipping, sampling steps, and the Ferminet architecture follow \cite{Goldshlager2024spring}.

% For the ANGD-Fisher-Rao algorithm, we further impose linear decay on  $\alpha_k, \beta_k$ with  $\alpha_0=h_0/2$, $\beta_0\in$ $\{0.01$, $0.05$, $0.1$, $0.15\}$, and decay rates for 
%  $\alpha_k$ in $\{1\text{e-4},$ $ 2\text{e-4}\}$, $\beta_k$ in $\{5\text{e-5}$, $2\text{e-4}\}$, respectively. The coefficient $\eta$ in \eqref{eq: solution of d_3} is set to 0.99 following the recommendation in \cite{Goldshlager2024spring}. 
%  For the ANGD/NGD-\W-2 algorithms,  we use the same hyperparameter selection as described above, but with some modifications for ANGD-\W-2: $\beta_k\equiv 0$,  $\alpha_0\in\{0.1/h_0,$ $0.2/h_0,$ $0.5/h_0\}$, and   $\epsilon\in\{0.0,$ $5\text{e-5}\}$.  Other hyper-parameters including   clipping, sampling steps, and the Ferminet architecture follow \cite{Goldshlager2024spring}.
 
% In addition, we apply several numerical stability techniques. We use the same norm constraint and damping coefficient for the NGD-Fisher-Rao algorithm as in \cite{Goldshlager2024spring}. For ANGD-Fisher-Rao, we adopt a restart strategy as turning $\overline{\Phi}_k$ to zero when $\|d_k\|>5$.
%   We apply coordinate-wise clipping with a threshold of $1.0$ to $\nabla\frac{\delta L}{\delta \rho_{\theta}}(x)$ for NGD-\W-2,  
%  and restart each coordinate of $V_k^i$ to $0.0$ if its magnitude exceeds $1.0$ in ANGD-\W-2. 

Comparisons of the performance between the ANGD and NGD algorithms on the four benchmark particles are shown in Figure \ref{fig:VMC}. We normalize the loss (energy) them by subtracting the  physical lower bound (reported in Hartrees to four decimal places). For both the Fisher-Rao and \W-2 metrics, the ANGD methods demonstrate significantly faster convergence rates and attain lower final losses than the NGD methods. Remarkably, even the worst-performing ANGD variant outperforms the best NGD variant in terms of final loss, highlighting the significant benefits of incorporating acceleration into the optimization process.

\section{Conclusion}
In this paper, we introduce a novel ANGD framework for solving parametrized manifold optimization problems. An ARG flow is designed to characterize accelerated optimization on a manifold, incorporating Hessian-driven damping. An equivalent system of first-order ODEs for several metrics is proposed. We develop a discretization scheme to project the ODE flow onto the parameter space, leading to an efficient solution of the accelerated direction. The convergence analysis of ARG flow under convexity assumptions is also established. Numerical experiments show that ANGD accelerates the optimization process compared to other methods.

\bibliographystyle{plain}
\bibliography{references}
\end{document}